\documentclass[11pt]{article}
\usepackage{amsfonts}
\usepackage{latexsym,amsmath,amsthm,amssymb}
\usepackage{mathrsfs,wasysym,graphicx,lscape}
\usepackage[pagewise]{lineno}
\begin{document}
\newtheorem{theorem}{ Theorem}[section]
\newtheorem{proposition}[theorem]{ Proposition}
\newtheorem{definition}[theorem]{Definition}
\newtheorem{lemma}{ Lemma}[section]
\newtheorem{remark}{ Remark}[section]
\newtheorem{Cor}[theorem]{Corollary}

\begin{center}
    {\large \bf The optimal global estimate and boundary behavior for large solutions to the $k$-Hessian equation}
\vspace{0.5cm}\\{\sc Haitao Wan, Yongxiu, Shi} \\
{\small School of Mathematics and Information Science, Shandong
Technology and Business University\\ Yantai, 264005, Shandong, China
}
\end{center}

\renewcommand{\theequation}{\arabic{section}.\arabic{equation}}
\numberwithin{equation}{section} \footnote[0]{\hspace*{-7.4mm}
Corresponding author: H. Wan\\
 E-mail: wht200805@163.com (H. Wan);
syxiu0926@126.com (Y. Shi)\\}

\renewcommand{\baselinestretch}{1.3}
\large\normalsize
\begin{abstract}
\noindent In this paper, we consider the $k$-Hessian equation
 $S_{k}(D^{2}u)=b(x)f(u)\mbox{ in }\Omega,\,u=+\infty
 \mbox{ on }\partial\Omega$, where $\Omega$ is a
smooth, bounded, strictly convex domain in $\mathbb{R}^{N}$ with
$N\geq2$, $b\in \rm C^{\infty}(\Omega)$ is positive in $\Omega$ and
may be singular or vanish on the boundary, $f\in C^{\infty}(0,
\infty)\cap C[0, +\infty)$ (or $f\in C^{\infty}(\mathbb{R})$) is
positive and increasing on $[0, +\infty)$ (or $\mathbb{R}$)  and
satisfies the Keller-Osserman type condition. We first supply an
upper and lower solution method of classical $k$-convex large
solutions to the above equation, and then we studied the optimal
global estimate and boundary behavior of large solutions. In
particular, we investigate the asymptotic behavior of such solutions
when the parameters on $b$ tend to the corresponding critical values
and infinity.

\noindent{\bf MSC2010: }35J25, 35J60, 35B40, 35J67.

\noindent {\bf Keywords:} $k$-Hessian equation; Boundary blow-up
problem; The upper and lower solution method; The optimal global
estimate; Asymptotic behavior.
\end{abstract}
\section{Introduction}
In this paper, we investigate the optimal global estimate and
boundary behavior of classical $k$-convex solutions to the following
$k$-Hessian problem
\begin{equation}\label{M}
S_{k}(D^{2}u)=S_{k}(\lambda_{1},\cdot\cdot\cdot,\lambda_{N})
=b(x)f(u)\mbox{ in }\Omega,\,u=+\infty \mbox{ on }\partial\Omega,
\end{equation}
where $\Omega\subset \mathbb{R}^{N}$ $(N\geq 2)$ is a smooth,
bounded, strictly convex domain, $\lambda_{1},\cdot\cdot\cdot,
\lambda_{N}$ are the eigenvalues of the Hessian
\[
D^{2}u(x)=\bigg(\frac{\partial^{2} u(x)}{\partial x_{i}\partial
x_{j}}\bigg)_{N\times N}
\]
of $u\in C^{2}(\Omega)$, the last condition $u=+\infty$ on
$\partial\Omega$ means that $u(x)\rightarrow+\infty$ as
$d(x):=\mbox{dist}(x,
\partial\Omega)\rightarrow0$ and the solution is called large solution, blow-up solution or explosive solution. The $b$ and $f$
satisfy
\begin{description}
\item[$\mathbf{(b_{1})}$] $b\in C^{\infty}(\Omega)$ is positive in
$\Omega$;
\item[$\mathbf{(S_{1})}$] $f\in C^{\infty}(0, +\infty)$ with
$f(0)=0$ and $f$ is strictly increasing on $[0, +\infty)$ (or
$\mathbf{(S_{01})}$ $f\in C^{\infty}(\mathbb{R}),$ $f(t)>0,
\forall\,t\in\mathbb{R}$, and $f$ is strictly increasing on
$\mathbb{R}$).
\end{description}

We see from \cite{C-N-S} and \cite{Trudinger} that
\[
S_{k}(\lambda_{1},  \cdot\cdot\cdot, \lambda_{N})=\sum_{1\leq
i_{1}<\cdot\cdot\cdot<i_{k}\leq
N}\lambda_{i_{1}}\cdot\cdot\cdot\lambda_{i_{k}}.
\]
For $k\in \{1, \cdot\cdot\cdot, N\}$, let $\Gamma_{k}$ be the
component of $\{\lambda\in\mathbb{R}^{N}: S_{k}(\lambda)>0\}\subset
\mathbb{R}^{N}$ contain the positive cone \[\Gamma^{+}:=\{\lambda\in
\mathbb{R}^{N}: \lambda_{i}>0, i=1,\cdot\cdot\cdot, N\},\] i.e.,
$\Gamma_{k}$ is the convex cone with vertex at the origin given by
\[
\Gamma_{k}:=\{\lambda\in \mathbb{R}^{N}: S_{i}(\lambda)>0,
i=1,\cdot\cdot\cdot, k\}.
\]
It follows from \cite{C-N-S2} that \[ \Gamma^{+}=\Gamma_{N}\subset
\cdot\cdot\cdot\subset \Gamma_{1}. \] From Definition 1.1 of
\cite{Salani}, we see that for $k\in\{1,\cdot\cdot\cdot,N\}$, $u\in
C^{2}(\Omega)$ is (strictly) $k$-convex if
\[
H_{i}(D^{2}u)=S_{i}(\lambda_{1},\cdot\cdot\cdot,\lambda_{N})\,(>)\geq
0 \mbox{ in }\Omega \mbox{ for }i=1,\cdot\cdot\cdot,k.
\]
Let $\Omega\subset \mathbb{R}^{N}$ be an open set with
$C^{2}$-boundary, then from Definition 1.2 of \cite{Salani}, we see
that, $\Omega$ is (strictly) convex if
\[
S_{i}(\kappa_{1},\cdot\cdot\cdot,\kappa_{N-1})\,(>)\geq0 \mbox{ on }
\partial\Omega \mbox{ for }i=1,\cdot\cdot\cdot, N-1,
\]
where $\kappa_{i}(x)$ $(i=1,\cdot\cdot\cdot,N-1)$ are the principal
curvatures of $\partial\Omega$ at $x$.

When $k=1$, problem \eqref{M} is the following semilinear elliptic
problem
\begin{equation}\label{M1}
\Delta u=b(x)f(u)\mbox{ in }\Omega,\,u=+\infty \mbox{ on
}\partial\Omega.
\end{equation}
 For $b\equiv1$ in $\Omega$, $f(u) = \exp(u)$ and $N = 2$, Bieberbach \cite{Bieberbach} first studied
the existence, uniqueness and asymptotic behavior of classical
solutions to problem \eqref{M1} in a bounded domain
$\Omega\subseteq\mathbb{R}^{N}$ with $C^{2}$-boundary.  Then,
Rademacher \cite{Rademacher} using the idea of Bieberbach, showed
the results still hold for $N=3$. Later, Lazer and McKenna
\cite{HLAZER} extended the above results in bounded domain $\Omega
\subset \mathbb{R}^{N}$ $(N\geq 1)$ with a uniform outer sphere
condition. In particular, Keller \cite{HKeller} and Osserman
\cite{HOsserman} carried out a systematic research on problem
\eqref{M1} and supplied a necessary and sufficient condition for the
existence of solutions to the problem. 
For further insight on problem (1.2), please refer to \cite{AGQ},
\cite{Bandle1}, \cite{Lazer-McKenna1}-\cite{HLAZER},
\cite{Diaz}-\cite{Garcia}, \cite{Gilbarg},
\cite{Loewner}-\cite{Marcus2},
 \cite{Veron} and the
references therein.

When $k=N$, problem \eqref{M} is the Monge-Amp\`{e}re problem
\begin{equation}\label{M2}
\mbox{det}(D^{2}u)=b(x)f(u) \mbox{ in }\Omega,\,u=+\infty \mbox{ on
}\partial\Omega.
\end{equation}
Problem \eqref{M2} arises from a few geometric problem and was
considered by Cheng and Yau \cite{HCheng-Yau2}-\cite{Cheng-Yau2} for
$f(u)=\exp(Ku)$ in bounded convex domains and for $f(u)=\exp(2u)$ in
unbounded domains. When $b\in C^{\infty}(\bar{\Omega})$ is positive
on $\bar{\Omega}$ and $f(u)=u^{\gamma}$ $(\gamma>n)$ or
$f(u)=\exp(u)$, Lazer and McKenna \cite{HLazer-McKenna} showed the
existence, uniqueness and global estimate of strictly convex
$C^{\infty}$-solutions to problem \eqref{M2}. In particular, they
showed that
\begin{description}
\item[$\mathbf{(i)}$]if $f(u)=u^{\gamma}$ with $\gamma>N$, then there exist constants $c_{1}, c_{2}>0$ such that the
strictly convex solution $u$ satisfies
\[
c_{1}(d(x))^{-(N+1)/(\gamma-N)}\leq u(x)\leq
c_{2}(d(x))^{-(N+1)/(\gamma-N)},\,x\in\Omega;
\]
\item[$\mathbf{(ii)}$]if $f(u)=\exp(u)$, then there exists constant $c_{3}>0$ such that the strictly convex
solution $u$ satisfies
\[
|u(x)+(N+1)\ln d(x)|<c_{3}.
\]
\end{description}
Moreover, they also proved the nonexistence when $f(u)=u^{\gamma}$
with $\gamma\in (0, N)$. Matero \cite{Matero} treated the more
general case for bounded strictly convex domains, generalizing a
result of Keller \cite{HKeller} and Osserman \cite{HOsserman}. Then,
Salani \cite{Salani} further extended the results to some
$k$-Hessian equations. Guan and Jian \cite{Guan-Jian} extended the
results of Cheng and Yau \cite{HCheng-Yau2}-\cite{Cheng-Yau2}, in
which various existence and nonexistence results were shown for
rather general Monge-Amp\`{e}re equations with gradient terms in
bounded or unbounded (strictly) convex domains. In particular, they
also studied the global estimate of strictly convex large solutions
to the problem in bounded strictly convex  domains. Then, the
results are extended by Jian \cite{Prof.Jian} to the case of
$k$-Hessian equations. Let $b$ satisfy $\mathbf{(b_{1})}$, $f$
satisfy $\mathbf{(S_{1})}$ (or $\mathbf{(S_{01})}$) and the
following Keller-Osserman type condition
\begin{equation}\label{inf1}
\int_{t}^{+\infty}((N+1)F(s))^{-1/(N+1)}ds<+\infty,\,t>0,\,F(t)=\int_{0}^{t}f(s)ds.
\end{equation}
Mohammed \cite{Mohammedproc} showed the existence of strictly convex
solutions to problem \eqref{M} when the Dirichlet problem
\begin{equation}\label{cr}
\mbox{ det }D^{2}u(x)=b(x),\,x\in\Omega,\,u|_{\partial\Omega}=0
\end{equation}
has a strictly convex solution $u_{0}\in C^{\infty}(\Omega)\cap
C(\bar{\Omega})$. Furthermore, when $
\int_{t}^{+\infty}(f(s))^{-1/N}ds=+\infty,$ the author showed
problem \eqref{M2} has no classical convex solution.  In fact, for
the studying of problem \eqref{cr} can be traced bake to the works
of Cheng and Yau in \cite{HCheng-Yau1} and  Caffarelli, Nirenberg
and Spruck in \cite{C-N-S2}. Cheng and Yau \cite{HCheng-Yau1} proved
that problem \eqref{cr} possesses a strictly convex solution if
$0<b(x)<C(d(x))^{\delta-N-1}$ in $\Omega$ for some constants
$\delta>0$ and $C>0$. In Theorem 1.1 of \cite{C-N-S2}, Caffarelli et
al. showed that problem \eqref{cr} admits a convex solution if $b\in
C^{\infty}(\bar{\Omega})$.  Mohammed \cite{MohammedJMAA2} showed
that if $b(x)>C(d(x))^{-N-1}$ in $\Omega$ with $C>0$, then problem
\eqref{cr} has no strictly convex solution. When $f(u)=\exp(u)$ or
$f(u)=u^{\gamma},\,\gamma>N$, the weight $b$ satisfy
$\mathbf{(b_{1})}$ and grows like a negative power of $d(x)$ near
boundary, Yang and Chang \cite{Yangh} showed the existence,
uniqueness, nonexistence and global estimate of strictly convex
solutions to problem \eqref{M2}, and when $\Omega$ is a ball, they
obtained the exact boundary behavior of large solutions. Recently,
Zhang and Du \cite{XZhang1} showed that if $b$ satisfy
$\mathbf{(b_{1})}$ and $b\in L^{\infty}(\Omega)$, $f$ satisfy
$\mathbf{(S_{1})}$ (or $\mathbf{(S_{01})}$) and
\begin{equation}\label{inf2}
\int_{0}^{1}(F(s))^{-1/(k+1)}ds=+\infty,\, F(t)=\int_{0}^{t}f(s)ds,
\end{equation}
then problem \eqref{M2} has a strictly convex $C^{\infty}$-solution
if and only if \eqref{inf1} holds, moreover, if
 $b$ satisfy $\mathbf{(b_{1})}$ and problem
\eqref{cr} admit a strictly convex solution, $f$ satisfy
$\mathbf{(S_{1})}$  $(\mbox{or } \mathbf{S_{01}})$, \eqref{inf1} and
\eqref{inf2}, then  problem \eqref{M2} has a strictly convex
$C^{\infty}$-solution. For other related works on Monge-Amp\`{e}re
equations and more general $k$-Hessian equations, please refer to
\cite{Bao}, \cite{Cirstea-Trom}-\cite{Colesanti},
\cite{GuangBo}-\cite{Ivochkina2}, \cite{Prof.Jian2}-\cite{Ji1},
\cite{Lions}, \cite{MaShan},  \cite{Plis},
\cite{Salani}-\cite{TrudingerWang}, \cite{ZhangxmF},
\cite{DSZhang1}-\cite{CPAA}.

Very recently, Zhang and Feng \cite{ZhangFeng1}, studied the
existence and global estimate of $k$-convex solutions to problem
\eqref{M} by using the following structure condition
\[
\lim_{t\rightarrow+\infty} J(t)=E_{f}^{+\infty}
\]
and
\begin{equation}\label{inf3}
\lim_{t\rightarrow0^{+}}J(t)=E_{f}^{0},
\end{equation}
where $E_{f}^{0}\in (0, +\infty)$, $E_{f}^{+\infty}\in (0, +\infty]$
and
\begin{equation}\label{JJJ}
J(t):=((F(t))^{1/(k+1)})'\int_{t}^{+\infty}(F(s))^{-1/(k+1)}ds.
\end{equation}
When $f(t)=\exp(t)$, we see by a direct calculation that $f(t)$ does
not satisfy \eqref{inf3}. In fact, \eqref{inf3} is just right when
$f$ satisfies $\mathbf{(S_{1})}$,  but it is inappropriate when $f$
satisfies $\mathbf{(S_{01})}$. Most recently, Zhang \cite{JFAZ} by
introduce some new local structure conditions established the
optimal global estimate and boundary behavior of strictly convex
solutions to problem \eqref{M2} when $f$ satisfies
$\mathbf{(S_{1})}$ (or $\mathbf{(S_{01})}$).

The aim of this paper is to extend and improve the main results of
\cite{JFAZ} to the case $k\in\{1,\cdot\cdot\cdot,N\}$. We first
supply the upper and lower solution method of classical $k$-convex
solution to problem \eqref{M}, and then we investigate the optimal
global estimate and boundary behavior of $k$-convex solutions to
problem \eqref{M}. In particular, in Theorems 2.4-2.5 of
\cite{JFAZ}, the author by structuring piecewise functions
established the global estimate of strictly convex  solutions to
problem \eqref{M2}. In this paper, we find that the results can be
improved by using Lemma \ref{Lemma4.2}. Moreover, we further study
the optimal global estimate of $k$-convex solutions to problem
\eqref{M} when $f$ is regularly varying at positive infinity with
the critical index $k$.

To obtain Theorems \ref{thm2.3}-\ref{thm2.5} and Theorem
\ref{thm2.8}, we assume that $f$ satisfies the following conditions,
not necessarily simultaneously:
\begin{description}
\item[$\mathbf{(f_{1})}$]$\int_{t}^{+\infty}(f(s))^{-1/k}ds<+\infty,\,t>0
\mbox{ or }t\in\mathbb{R}$ and
\[
I(t):=((f(t))^{1/k})'\int_{t}^{+\infty}(f(s))^{-1/k}ds,\,t>0 \mbox{
or }t\in\mathbb{R};
\]
\item[$\mathbf{(f_{2})}$] there exists $C_{f}^{+\infty}\in (0,
+\infty)$ such that
\[
\lim_{t\rightarrow+\infty}I(t)=C_{f}^{+\infty};
\]
\item[$\mathbf{(f_{3})}$] $f$ satisfies $\mathbf{(S_{1})}$,
$\int_{0}^{1}(f(s))^{-1/k}ds=+\infty$, and there exists $C_{f}^{0}$
such that
\[
\lim_{t\rightarrow0^{+}}I(t)=C_{f}^{0};
\]
\item[$\mathbf{(f_{4})}$] $f$ satisfies $\mathbf{(S_{01})}$, there
exists $C_{f}^{-\infty}\in (0, +\infty)$ such that
\[
\lim_{t\rightarrow-\infty}I(t)=C_{f}^{-\infty}.
\]
\end{description}
By $\mathbf{(S_{01})}$, we see that
\[
\int_{-\infty}^{a}(f(s))^{-1/k}ds=+\infty.
\]
\begin{remark}
Some basic examples of $f$ in
$\mathbf{(f_{1})}$-$\mathbf{(f_{3})}$$(\mbox{or } \mathbf{(f_{4})})$
are
\item[$\mathbf{(i)}$]if $f(t)=t^{\gamma}$, $t\geq0$ with $\gamma>k$,
then $C_{f}^{+\infty}=C_{f}^{0}=\frac{\gamma}{\gamma-k}$;
\item[$\mathbf{(ii)}$]if $f(t)=\exp(t),\,t\in\mathbb{R}$, then
$C_{f}^{+\infty}=C_{f}^{-\infty}=1$;
\item[$\mathbf{(iii)}$] if $f(t)\sim (-t)^{\gamma}$
as $t\rightarrow-\infty$, where $\gamma<0$, then
$C_{f}^{-\infty}=\frac{\gamma}{\gamma-1}$.
\end{remark}

To obtain Theorem \ref{thm2.6}, we assume that $f$ satisfies the
following conditions, not necessarily simultaneously:
\begin{description}
\item[$\mathbf{(f_{5})}$] $\int_{t}^{+\infty}(F(s))^{-1/(k+1)}ds<+\infty,\,t>0 \mbox{ or }
t\in\mathbb{R}$, where
\[
F(t)=\int_{0}^{t}f(s)ds \mbox{ or }F(t)=\int_{-\infty}^{t}f(s)ds;
\]
\item[$\mathbf{(f_{6})}$] $\lim_{t\rightarrow+\infty} J(t)=+\infty$, where $J$ is given by \eqref{JJJ};
\item[$\mathbf{(f_{7})}$] $f$ satisfies $\mathbf{(S_{1})}$,
\eqref{inf2} and there exists $E_{f}^{0}\in (0, +\infty)$ such that
\eqref{inf3} holds;
\item[$\mathbf{(f_{8})}$] $f$ satisfies $\mathbf{(S_{01})}$, there
exists $E_{f}^{-\infty}\in (0, +\infty)$ such that
\[
\lim_{t\rightarrow-\infty}J(t)=E_{f}^{-\infty},\mbox{ where
}F(t)=\int_{-\infty}^{t}f(s)ds.
\]
\end{description}
By $\mathbf{(S_{01})}$, we see that
\[
\int_{-\infty}^{a}(F(s))^{-1/(k+1)}ds=+\infty \mbox{ for
}a\in\mathbb{R}.
\]
\begin{remark}
Some basic examples of $f$ in $\mathbf{(f_{5})}$, $\mathbf{(f_{7})}$
$(\mbox{or } \mathbf{(f_{8})})$ are
\item[$\mathbf{(i)}$] if $f(t)\sim t^{\gamma}$ as $t\rightarrow0^{+}$, where
$\gamma>k$, then $E_{f}^{0}=\frac{\gamma+1}{\gamma-k}$;
\item[$\mathbf{(ii)}$]if $f(t)\sim(-t)^{\gamma}$ as
$t\rightarrow-\infty$, where $\gamma<-1$, then
$E_{f}^{-\infty}=\frac{\gamma+1}{\gamma-k}$;
\item[$\mathbf{(iii)}$]if $f(t)\sim \exp(t)$ as
$t\rightarrow-\infty$, then $E_{f}^{-\infty}=1$.
\end{remark}
\section{Main Results}
\subsection{Optimal global behavior}
$\mathbf{(I)}$ Let $\psi$ be uniquely determined by
\begin{equation}\label{f22}
\int_{\psi(t)}^{+\infty}(f(s))^{-1/k}ds=t,\,t>0.
\end{equation}
We note that
\begin{description}
\item[$\mathbf{(i)}$] $\mathbf{(S_{1})}$ (or $\mathbf{S_{01}}$),
$\mathbf{(f_{1})}$-$\mathbf{(f_{2})}$ imply
\[
\psi(t)\rightarrow+\infty \mbox{ if and only if }t\rightarrow0^{+};
\]
\item[$\mathbf{(ii)}$] $\mathbf{(S_{1})}$, $\mathbf{(f_{1})}$ and
$\mathbf{(f_{3})}$ imply
\[
\psi(t)\rightarrow0^{+} \mbox{ if and only if }t\rightarrow+\infty;
\]
\item[$\mathbf{(iii)}$] $\mathbf{(S_{01})}$, $\mathbf{(f_{1})}$ and
$\mathbf{(f_{4})}$ imply
\[
\psi(t)\rightarrow-\infty \mbox{ if and only if
}t\rightarrow+\infty.
\]
\end{description}
$\mathbf{(II)}$ Since $\Omega$ is smooth, bounded, strictly convex
domain in $\mathbb{R}^{N}$, we see from \cite{HLazer-McKenna} that
there exists a function $\phi\in C^{\infty}(\bar{\Omega})$ with the
following properties:
\[
\phi(x)<0,\,x\in\Omega,\,\phi|_{\partial\Omega}=0,\,\nabla
\phi(x)\neq 0,\,x\in\partial\Omega
\]
and $\phi$ is positive definite in $\bar{\Omega}$. Let $v=-\phi$ and
assume that
\[
\max_{x\in\bar{\Omega}} v(x)<1.
\]
It is clear that $D^{2}v$ is negative definite in $\bar{\Omega}$. In
fact, for any positive constant  $\rho\in(0, 1]$, we can always take
$\phi$ such that
\[
\max_{x\in\bar{\Omega}} v(x)<\rho.
\]
$\mathbf{(III)}$ Let $\mathcal {L}$ denote the set of Karamata
functions defined on $(0, 1]$ by
\[
\tilde{L}(t)=c\exp\bigg(\int_{t}^{1}\frac{y(s)}{s}ds\bigg),\,c>0,\,y\in
C(0, 1] \mbox{ and }y(t)\rightarrow0 \mbox{ as }t\rightarrow0^{+}.
\]
We see from Proposition \ref{proposition4.6} that $\tilde{L}$ is
normalized slowly varying at zero.

Our results can be summarized as follows.
\begin{theorem}\label{thm2.1}
Let $f(t)=t^{\gamma},$ $t\geq0$ with $\gamma>k$, $b$ satisfy
$\mathbf{(b_{1})}$ and
\begin{description}
\item[$\mathbf{(b_{2})}$] there exist positive constants  $b_{1}<b_{2}$
and $\lambda>-1-k$ such that
\[
b_{1}(v(x))^{\lambda}\leq b(x)\leq b_{2}(v(x))^{\lambda},
\]
\end{description}
then problem \eqref{M} has a unique classical $k$-convex solution
$u_{\lambda}$ satisfying
\[
m_{0} (v(x))^{-\alpha}\leq u_{\lambda}(x)\leq M_{0}
(v(x))^{-\alpha},\,x\in\Omega,
\]
where
\begin{equation}\label{fa}
\alpha=\frac{k+1+\lambda}{\gamma-k},\,\,\,
m_{0}=\bigg(\frac{b_{2}}{\alpha^{k}c_{0}}\bigg)^{1/(k-\gamma)},\,\,\,M_{0}=\bigg(\frac{b_{1}}{\alpha^{k}C_{0}}\bigg)^{1/(k-\gamma)}
\end{equation}
with
\[
c_{0}=c_{0}(\lambda)=\min_{x\in\bar{\Omega}}\omega_{0}(x),\,\,\,C_{0}=C_{0}(\lambda)=\max_{x\in\bar{\Omega}}\omega_{0}(x),
\]
where
\begin{equation}\label{fb}
\omega_{0}(x)=v(x)\cdot(-1)^{k}S_{k}(D^{2}v(x))+(\alpha+1)\sum_{i=1}^{C_{N}^{k}}(-1)^{k}det(v_{x_{is}x_{ij}})(-\nabla
v_{i}(x))^{T}B(v_{i}(x))\nabla v_{i}(x),
\end{equation}
where $B(v_{i})$ denotes the inverse of the $i$-th principal
submatrix $(v_{x_{is}x_{ij}})$, $\mbox{det}(v_{x_{is}x_{ij}})$
denotes  the determinant of $(v_{x_{is}x_{ij}})$ and
\[
\nabla v_{i}=(v_{x_{i1}}, \cdot\cdot\cdot,
v_{x_{ik}})^{T},\,i=1,\cdot\cdot\cdot, C_{N}^{k} \mbox{ and }
C_{N}^{k}:=\frac{N!}{(N-k)!k!}.
\]
Moreover, we have
\begin{equation}\label{f8}
\lim_{\lambda\rightarrow-k-1}u_{\lambda}(x)=0,\,\,\,\lim_{\lambda\rightarrow+\infty}u_{\lambda}(x)=+\infty
\end{equation}
and
\begin{equation}\label{f9}
\bigg(\frac{b_{2}}{c_{1}}\bigg)^{1/(k-\gamma)}\leq\liminf_{\lambda\rightarrow-k-1}\frac{u_{\lambda}(x)}{\alpha^{k/(\gamma-k)}}\leq
\limsup_{\lambda\rightarrow
-k-1}\frac{u_{\lambda}(x)}{\alpha^{k/(\gamma-k)}}\leq
\bigg(\frac{b_{1}}{C_{1}}\bigg)^{1/(k-\gamma)}
\end{equation}
uniformly for $x\in \Omega_{1}$ which is an arbitrary compact subset
of $\Omega$, where
\begin{equation}\label{f1}
c_{1}=\min_{x\in\bar{\Omega}}\omega_{1}(x),\,\,\,C_{1}=\max_{x\in\bar{\Omega}}\omega_{1}(x)
\end{equation}
with
\begin{equation}\label{2J1}
\omega_{1}(x)=v(x)(-1)^{k}S_{k}(D^{2}v(x))+\sum_{i=1}^{C_{N}^{k}}(-1)^{k}det(v_{x_{is}x_{ij}}(x))(-\nabla
v_{i}(x))^{T}B(v_{i}(x))\nabla v_{i}(x).
\end{equation}
In particular, for fixed $x\in\Omega$, we have
\[
b_{2}^{1/(k-\gamma)}\leq\liminf_{\lambda\rightarrow
+\infty}\frac{u_{\lambda}(x)}{\alpha^{k/(\gamma-k)}c_{0}^{1/(\gamma-k)}(v(x))^{-\alpha}}\mbox{
and }
\limsup_{\lambda\rightarrow+\infty}\frac{u_{\lambda}(x)}{\alpha^{k/(\gamma-k)}C_{0}^{1/(\gamma-k)}(v(x))^{-\alpha}}\leq
b_{1}^{1/(k-\gamma)}.
\]
\end{theorem}
\begin{theorem}\label{thm2.2}
Let $f(t)=\exp(t)$, $t\in\mathbb{R}$, $b$ satisfy
$\mathbf{(b_{1})}$-$\mathbf{(b_{2})}$, then problem \eqref{M} has a
unique classical $k$-convex solution $u_{\lambda}$ satisfying
\[
m_{1}-(\lambda+k+1)\ln v(x)\leq u_{\lambda}(x)\leq
M_{1}-(\lambda+k+1)\ln v(x),\,x\in\Omega,
\]
where
\begin{equation}\label{f13}
m_{1}=\ln c_{1}+k\ln(\lambda+k+1)-\ln b_{2},\,\,\,M_{1}=\ln
C_{1}+k\ln (\lambda+k+1)-\ln b_{1}
\end{equation}
and $c_{1}$, $C_{1}$ are given by \eqref{f1}. Moreover, we have
\begin{equation}\label{f14}
\lim_{\lambda\rightarrow-k-1}u_{\lambda}(x)=-\infty,\,\,\,
\lim_{\lambda\rightarrow+\infty}u_{\lambda}(x)=+\infty
\end{equation}
and
\begin{equation}\label{f15}
\lim_{\lambda\rightarrow-k-1}\frac{u_{\lambda}(x)}{\ln
(k+1+\lambda)}=k,\,\,\,\lim_{\lambda\rightarrow+\infty}\frac{u_{\lambda}(x)}{(k+1+\lambda)\ln
v(x)}=-1
\end{equation}
 uniformly for $x\in \Omega_{1}$ which is an arbitrary compact
subset of $\Omega$. 
\end{theorem}
\begin{theorem}\label{thm2.3}
Let $f$ satisfy $\mathbf{(S_{1})}$ $(\mbox{or }\mathbf{(S_{01})})$,
$\mathbf{(f_{1})}$-$\mathbf{(f_{2})}$ and $\mathbf{(f_{3})}$
$(\mbox{or }\mathbf{(f_{4})})$, $b$ satisfy
$\mathbf{(b_{1})}$-$\mathbf{(b_{2})}$ with
\begin{equation}\label{f11}
(h_{0}-1)\eta+1>0,
\end{equation}
then problem \eqref{M} has a $k$-convex solution $u_{\lambda}$
satisfying
\begin{equation}\label{4f3}
\psi(\tau_{2}\eta^{-1}v^{\eta}(x))\leq u_{\lambda}(x)\leq
\psi(\tau_{1}\eta^{-1}v^{\eta}(x)),\,x\in\Omega,
\end{equation}
where
\begin{equation}\label{ftj}
\eta=\frac{k+1+\lambda}{k},\,\,\,h_{0}:=\begin{cases}
\inf\limits_{t>0}((f(t))^{1/k})'\int_{t}^{+\infty}(f(s))^{-1/k}ds>0,\, &\mbox{ if } \mathbf{(f_{3})} \mbox{ holds},\\
\inf\limits_{t\in\mathbb{R}}((f(t))^{1/k})'\int_{t}^{+\infty}(f(s))^{-1/k}ds>0,\,
&\mbox{ if } \mathbf{(f_{4})} \mbox{ holds},
\end{cases}
\end{equation}
 $\tau_{1}$ and $\tau_{2}$ are given by
\begin{equation}\label{f12}
\tau_{1}^{k}\max_{x\in\bar{\Omega}}\omega_{2}(\tau_{1},
x)=b_{1},\,\,\,\tau_{2}^{k}\min_{x\in\bar{\Omega}}\omega_{2}(\tau_{2},
x)=b_{2}
\end{equation}
with
\begin{equation}\label{2J2}
\begin{split}
\omega_{2}(\tau_{j},
x)&=v(x)(-1)^{k}S_{k}(D^{2}v(x))\\
&+\big(\Psi(\tau_{j}\eta^{-1}(v(x))^{\eta})\eta+1-\eta\big)\sum_{i=1}^{C_{N}^{k}}(-1)^{k}\det(v_{x_{is}x_{ij}}(x))(-\nabla
v_{i}(x))^{T}B(v_{i}(x))\nabla v_{i}(x)
\end{split}
\end{equation}
and
\[
\Psi(\tau_{j}\eta^{-1}(v(x))^{\eta})=-\frac{\psi''(\tau_{j}\eta^{-1}(v(x))^{\eta})\tau_{j}\eta^{-1}(v(x))^{\eta}}{\psi'(\tau_{j}\eta^{-1}(v(x))^{\eta})},\,j=1,2.
\]
Moreover, we have
\[
\begin{cases}
\lim\limits_{\lambda\rightarrow -k-1}\max\limits_{x\in \bar{\Omega}_{1}}u_{\lambda}(x)=0,\, &\mbox{ if } \mathbf{(f_{3})} \mbox{ holds},\\
\lim\limits_{\lambda\rightarrow -k-1}\max\limits_{x\in
\bar{\Omega}_{1}}u_{\lambda}(x)=-\infty,\, &\mbox{ if }
\mathbf{(f_{4})} \mbox{ holds};\\
\lim\limits_{\lambda\rightarrow+\infty}\min\limits_{x\in\Omega_{1}}u_{\lambda}(x)=+\infty,\,
&\mbox{ if } \mathbf{(f_{3})}\, (\mbox{or }\mathbf{(f_{4})} \mbox{
with }C_{f}^{-\infty}=1)  \mbox{ holds}.
\end{cases}
\]
In particular, if $\mathbf{(f_{3})}$ holds, then we further have
\begin{equation}\label{f2}
\bigg(\frac{b_{2}}{c_{1}}\bigg)^{(1-C_{f}^{0})/k}\leq\liminf_{\lambda\rightarrow-k-1}\frac{u_{\lambda}(x)}{\psi(\eta^{-1})}\leq\limsup_{\lambda\rightarrow-k-1}\frac{u_{\lambda}(x)}{\psi(\eta^{-1})}\leq\bigg(\frac{b_{1}}{C_{1}}\bigg)^{(1-C_{f}^{0})/{k}}
\end{equation}
uniformly for $x\in \Omega_{1}$ which is an arbitrary compact subset
of $\Omega$, and if $\mathbf{(f_{4})}$ holds, then we further have
\begin{equation}\label{f3}
\bigg(\frac{b_{1}}{C_{1}}\bigg)^{(1-C_{f}^{-\infty})/k}\leq\liminf_{\lambda\rightarrow-k-1}\frac{u_{\lambda}(x)}{\psi(\eta^{-1})}\leq\limsup_{\lambda\rightarrow-k-1}\frac{u_{\lambda}(x)}{\psi(\eta^{-1})}\leq\bigg(\frac{b_{2}}{c_{1}}\bigg)^{(1-C_{f}^{-\infty})/{k}}
\end{equation}
uniformly for $x\in \Omega_{1}$.

It follows from \eqref{f2} and \eqref{f3} that if $C_{f}^{0}=1$
$($or $C_{f}^{-\infty}=1$$)$, then for any $x\in\Omega$,
\[
\lim_{\lambda\rightarrow-k-1}\frac{u_{\lambda}(x)}{\psi(\eta^{-1})}=1.
\]
\end{theorem}
\begin{remark}
If $\mathbf{(f_{4})}$ holds, then by Lemma \ref{lemma23}
$\mathbf{(i)}$, we have $h_{0}\leq1$. So, in Theorem \ref{thm2.3},
$\mathbf{(f_{4})}$ with $C_{f}^{-\infty}=1$ and \eqref{f11} for any
$\eta>0$ imply that $h_{0}=1$.
\end{remark}
\begin{theorem}\label{thm2.4}
Let $f$ satisfy $\mathbf{(S_{1})}$ $(\mbox{or }\mathbf{(S_{01})})$,
$\mathbf{(f_{1})}$-$\mathbf{(f_{2})}$ and $\mathbf{(f_{3})}$
$(\mbox{or }\mathbf{(f_{4})})$, $b$ satisfy $\mathbf{(b_{1})}$ and
the following condition
\begin{description}
\item[$\mathbf{(b_{3})}$] there exist $\mu>1$ and positive constants
$b_{1}<b_{2}$ such that for any $x\in \Omega$,
\[
b_{1}(v(x))^{-k-1}(-\ln v(x))^{-k\mu}\leq b(x)\leq
b_{2}(v(x))^{-k-1}(-\ln v(x))^{-k\mu},
\]
\end{description}
then problem \eqref{M} has a classical $k$-convex solution $u_{\mu}$
satisfying
\begin{equation}\label{4f2}
\psi(\tau_{4}(\mu-1)^{-1}(-\ln v(x))^{1-\mu})\leq u_{\mu}(x)\leq
\psi(\tau_{3}(\mu-1)^{-1}(-\ln v(x))^{1-\mu}),\,x\in\Omega,
\end{equation}
where  $\tau_{3}$ and $\tau_{4}$ are given by
\begin{equation}\label{f16}
\tau_{3}^{k}\max_{x\in\bar{\Omega}}\omega_{3}(\tau_{3},
x)=b_{1},\,\,\,\tau_{4}^{k}\min_{x\in\bar{\Omega}}\omega_{3}(\tau_{4},
x)=b_{2}
\end{equation}
with
\begin{equation}\label{2J3}
\begin{split}
\omega_{3}(\tau_{j},
x)&=v(x)(-1)^{k}S_{k}(D^{2}v(x))+\big(\Psi(\tau_{j}(\mu-1)^{-1}(-\ln
v(x))^{1-\mu})(\mu-1)(-\ln v(x))^{-1}\\
&+1-\mu(\ln
v(x))^{-1}\big)\sum_{i=1}^{C_{N}^{k}}(-1)^{k}\mbox{det}(v_{x_{is}x_{ij}}(x))(-\nabla
v_{i}(x))^{T}B(v_{i}(x))\nabla v_{i}(x)
\end{split}
\end{equation}
and
\[
\begin{split}
&\Psi(\tau_{j}(\mu-1)^{-1}(-\ln
v(x))^{1-\mu})\\
&=-\frac{\psi''(\tau_{j}(\mu-1)^{-1}(-\ln
v(x))^{1-\mu})\tau_{j}(\mu-1)^{-1}(-\ln
v(x))^{1-\mu})}{\psi'(\tau_{j}(\mu-1)^{-1}(-\ln
v(x))^{1-\mu})},\,j=3,4.
\end{split}
\]
Moreover, we have
\[
\begin{cases}
\lim\limits_{\mu\rightarrow 1^{+}}\max\limits_{x\in \bar{\Omega}_{1}}u_{\mu}(x)=0,\, &\mbox{ if } \mathbf{(f_{3})} \mbox{ holds},\\
\lim\limits_{\mu\rightarrow 1^{+}}\max\limits_{x\in
\bar{\Omega}_{1}}u_{\mu}(x)=-\infty,\, &\mbox{ if }
\mathbf{(f_{4})} \mbox{ holds};\\
\lim\limits_{\mu\rightarrow+\infty}\min\limits_{x\in\bar{\Omega}_{1}}u_{\mu}(x)=+\infty,\,
&\mbox{ if } \mathbf{(f_{3})} \mbox{ or }\mathbf{(f_{4})} \mbox{
holds}.
\end{cases}
\]
In particular, if $\mathbf{(f_{3})}$ holds, then
\begin{equation}\label{f4}
\bigg(\frac{b_{2}}{c_{2}}\bigg)^{(1-C_{f}^{0})/k}\leq\liminf_{\mu\rightarrow
1^{+}}\frac{u_{\mu}(x)}{\psi((\mu-1)^{-1})}\leq\limsup_{\mu\rightarrow
1^{+}}\frac{u_{\mu}(x)}{\psi((\mu-1)^{-1})}\leq\bigg(\frac{b_{1}}{C_{2}}\bigg)^{(1-C_{f}^{0})/{k}}
\end{equation}
uniformly for $x\in \Omega_{1}$ which is an arbitrary compact subset
of $\Omega$, and if $\mathbf{(f_{4})}$ holds, then we further have
\begin{equation}\label{f5}
\bigg(\frac{b_{1}}{C_{2}}\bigg)^{(1-C_{f}^{-\infty})/k}\leq\liminf_{\mu\rightarrow1^{+}}\frac{u_{\mu}(x)}{\psi((\mu-1)^{-1})}\leq\limsup_{\mu\rightarrow1^{+}}\frac{u_{\mu}(x)}{\psi((\mu-1)^{-1})}\leq\bigg(\frac{b_{2}}{c_{2}}\bigg)^{(1-C_{f}^{-\infty})/{k}}
\end{equation}
uniformly for $x\in \Omega_{1}$, where
\[
c_{2}=\min_{x\in\bar{\Omega}}\hat{\omega}_{2}(x),\,\,\,C_{2}=\max_{x\in\bar{\Omega}}\hat{\omega}_{2}(x)
\]
with
\[
\begin{split}
\hat{\omega}_{2}(x)&=v(x)(-1)^{k}S_{k}(D^{2}v(x))\\
&+\big(1-(\ln
v(x))^{-1}\big)\sum_{i=1}^{C_{N}^{k}}(-1)^{k}\mbox{det}(v_{x_{is}x_{ij}}(x))(-\nabla
v_{i}(x))^{T}B(v_{i}(x))\nabla v_{i}(x).
\end{split}
\]

It follows from \eqref{f4} and \eqref{f5} that if $C_{f}^{0}=1$
$($or $C_{f}^{-\infty}=1$$)$, then for any $x\in\Omega$,
\[
\lim_{\mu\rightarrow1^{+}}\frac{u_{\mu}(x)}{\psi((\mu-1)^{-1})}=1.
\]
\end{theorem}
\begin{theorem}\label{thm2.5}
Let $f$ satisfy $\mathbf{(S_{1})}$ $(\mbox{or }\mathbf{(S_{01})})$,
$\mathbf{(f_{1})}$-$\mathbf{(f_{2})}$ and $\mathbf{(f_{3})}$
$(\mbox{or }\mathbf{(f_{4})})$, $b$ satisfy $\mathbf{(b_{1})}$ and
the following condition
\begin{description}
\item[$\mathbf{(b_{4})}$] there exist positive constants $b_{1}<
b_{2}$, $\lambda\geq -k-1$ and some function $\tilde{L}\in \mathcal
{L}$ such that
\[
b_{1}(v(x))^{\lambda}\tilde{L}^{k}(v(x))\leq b(x)\leq
b_{2}(v(x))^{\lambda}\tilde{L}^{k}(v(x)).
\]
\end{description}
If we further assume that
\begin{equation}\label{f17}
kh_{0}+(1+\lambda)(h_{0}-1)
\end{equation}
in $\mathbf{(b_{4})}$, where $h_{0}$ is given by \eqref{ftj}, then
problem \eqref{M} has a classical $k$-convex solution $u$ satisfying
\begin{equation}\label{4f1}
\psi\bigg(\tau_{6}\int_{0}^{v(x)}s^{\frac{1+\lambda}{k}}\tilde{L}(s)ds\bigg)\leq
u(x)\leq
\psi\bigg(\tau_{5}\int_{0}^{v(x)}s^{\frac{1+\lambda}{k}}\tilde{L}(s)ds\bigg),\,x\in\Omega,
\end{equation}
where  $\tau_{5}$ and $\tau_{6}$ are given by
\begin{equation}\label{f18}
\tau_{5}^{k}\max_{x\in\bar{\Omega}}\omega_{4}(\tau_{5},
x)=b_{1},\,\,\,\tau_{6}^{k}\min_{x\in
\bar{\Omega}}\omega_{4}(\tau_{6}, x)=b_{2}
\end{equation}
with
\begin{equation}\label{2J4}
\begin{split}
\omega_{4}(\tau_{j},
x)&=v(x)(-1)^{k}S_{k}(D^{2}v(x))+\bigg[\Psi\bigg(\tau_{j}\int_{0}^{v(x)}s^{\frac{1+\lambda}{k}}\tilde{L}(s)ds\bigg)\frac{(v(x))^{\frac{k+1+\lambda}{k}}\tilde{L}(v(x))}{\int_{0}^{v(x)}s^{\frac{1+\lambda}{k}}\tilde{L}(s)ds}\\
&-\frac{\lambda+1}{k}-\frac{v(x)\tilde{L}'(v(x))}{\tilde{L}(v(x))}\bigg]\sum_{i=1}^{C_{N}^{k}}(-1)^{k}det(
v_{x_{is}x_{ij}})(-\nabla v_{i}(x))^{T}B(v_{i}(x))\nabla v_{i}(x)
\end{split}
\end{equation}
and
\[
\Psi\bigg(\tau_{j}\int_{0}^{v(x)}s^{\frac{1+\lambda}{k}}\tilde{L}(s)ds\bigg)=-\frac{\psi''(\tau_{j}\int_{0}^{v(x)}s^{\frac{1+\lambda}{k}}\tilde{L}(s)ds)\tau_{j}\int_{0}^{v(x)}s^{\frac{1+\lambda}{k}}\tilde{L}(s)ds}{\psi'(\tau_{j}\int_{0}^{v(x)}s^{\frac{1+\lambda}{k}}\tilde{L}(s)ds)},
j=5,6.
\]
\end{theorem}
\begin{remark}
In Theorem \ref{thm2.5}, if $\lambda=-k-1$, we need verify
\[\int_{0}^{1}\frac{\tilde{L}(s)}{s}ds<+\infty.\]
\end{remark}
\begin{remark}
In Theorem \ref{thm2.5}, if $\lambda=-k-1$ and
$\tilde{L}(v(x))=(-\ln v(x))^{k\mu}$ with $\mu>1$, then the global
estimate \eqref{4f1} is the same as \eqref{4f2}. If $\lambda>-k-1$
and $\tilde{L}\equiv 1$, then the estimate \eqref{4f1} is the same
as \eqref{4f3}.
\end{remark}
\begin{remark}
In Theorem \ref{thm2.5}, if $f(u)=u^{\gamma}$ with $\gamma>k$, then
the $k$-convex solution $u$ satisfies
\[
\bigg(\frac{(\gamma-k)\tau_{6}}{k}\int_{0}^{v(x)}s^{\frac{1+\lambda}{k}}\tilde{L}(s)ds\bigg)^{\frac{k}{k-\gamma}}\leq
u(x)\leq
\bigg(\frac{(\gamma-k)\tau_{5}}{k}\int_{0}^{v(x)}s^{\frac{1+\lambda}{k}}\tilde{L}(s)ds\bigg)^{\frac{k}{k-\gamma}},\,x\in\Omega,
\]
and if $f(u)=\exp(u)$, then the $k$-convex solution $u$ satisfies
\[
-k\bigg[\ln\bigg(\tau_{6}\int_{0}^{v(x)}s^{\frac{1+\lambda}{k}}\tilde{L}(s)ds\bigg)-\ln
k\bigg]\leq
u(x)\leq-k\bigg[\ln\bigg(\tau_{5}\int_{0}^{v(x)}s^{\frac{1+\lambda}{k}}\tilde{L}(s)ds\bigg)-\ln
k\bigg],\,x\in\Omega.
\]
\end{remark}

In fact, by Lemma \ref{lemma21} $\mathbf{(ii)}$-$\mathbf{(iii)}$, we
see that the conditions $\mathbf{(S_{1})}$ (or $\mathbf{(S_{01})}$),
$\mathbf{(f_{1})}$-$\mathbf{(f_{2})}$ imply that $f\in NRV_{\gamma}$
with $\gamma>k$ or $f$ is rapidly varying to positive infinity at
positive infinity. Next, we will show the optimal global behavior of
$k$-convex solutions to problem \eqref{M} when $f\in RV_{k}$.

Let $\varphi$ be uniquely determined by
\begin{equation}\label{f23}
\int_{\varphi(t)}^{+\infty}((k+1)F(s))^{-1/(k+1)}ds=t,
\end{equation}
where
\begin{equation}\label{f21}
F(t)=\int_{\varsigma}^{t}f(s)ds \mbox{ with }\varsigma=\begin{cases}
0,\, &\mbox{ if } \mathbf{(S_{1})} \mbox{ holds},\\
-\infty,\,&\mbox{ if } \mathbf{(S_{01})} \mbox{ holds}.
\end{cases}
\end{equation}
We note that
\begin{description}
\item[$\mathbf{(i)}$] $\mathbf{(S_{1})}$ (or $\mathbf{(S_{01})}$),
$\mathbf{(f_{5})}$ imply
\[
\varphi(t)\rightarrow+\infty \mbox{ if and only if
}t\rightarrow0^{+};
\]
\item[$\mathbf{(ii)}$]$\mathbf{(S_{1})}$, $\mathbf{(f_{5})}$ and
$\mathbf{(f_{7})}$ imply
\[
\varphi(t)\rightarrow0^{+} \mbox{ if and only if
}t\rightarrow+\infty;
\]
\item[$\mathbf{(iii)}$]$\mathbf{(S_{01})}$, $\mathbf{(f_{5})}$ and
$\mathbf{(f_{8})}$ imply
\[
\varphi(t)\rightarrow-\infty \mbox{ if and only if
}t\rightarrow+\infty.
\]
\end{description}

Our result can be summarized as follows.
\begin{theorem}\label{thm2.6}
Let $f$ satisfy $\mathbf{(S_{1})}$ $\mathbf{(\mbox{or }S_{01})}$,
$\mathbf{(f_{5})}$-$\mathbf{(f_{6})}$ and $\mathbf{(f_{7})}$ $(\mbox{or }\mathbf{(f_{8})}$, $b$ satisfy
$\mathbf{(b_{1})}$ and $\mathbf{(b_{4})}$ with $-k-1<\lambda<0$,
then problem \eqref{M} has a classical $k$-convex solution $u$
satisfying
\[
\varphi\bigg(\tau_{8}\bigg(\int_{0}^{v(x)}s^{\frac{1+\lambda}{k}}\tilde{L}(s)ds\bigg)^{\frac{k}{k+1}}\bigg)\leq
u(x)\leq
\varphi\bigg(\tau_{7}\bigg(\int_{0}^{v(x)}s^{\frac{1+\lambda}{k}}\tilde{L}(s)ds\bigg)^{\frac{k}{k+1}}\bigg),\,x\in\Omega,
\]
where $\tau_{7}$ and $\tau_{8}$ are given by
\begin{equation*}\label{f20}
\tau_{7}^{k+1}\max_{x\in\bar{\Omega}}\omega_{5}(\tau_{7},
x)=b_{1},\,\,\,\tau_{8}^{k+1}\min_{x\in\bar{\Omega}}\omega_{5}(\tau_{8}
x)=b_{2},
\end{equation*}
with
\begin{equation}\label{2J5}
\begin{split}
\omega_{5}(\tau_{j}, x)&=\bigg(\frac{k}{k+1}\bigg)^{k}\bigg\{\Phi\bigg(\tau_{j}\bigg(\int_{0}^{v(x)}s^{\frac{1+\lambda}{k}}\tilde{L}(s)ds\bigg)^{\frac{k}{k+1}}\bigg)v(x)(-1)^{k}S_{k}(D^{2}v(x))\\
&+\bigg[\frac{(v(x))^{\frac{k+1+\lambda}{k}}\tilde{L}(v(x))}{\int_{0}^{v(x)}s^{\frac{1+\lambda}{k}}\tilde{L}(s)ds}+\Phi\bigg(\tau_{j}\bigg(\int_{0}^{v(x)}s^{\frac{1+\lambda}{k}}\tilde{L}(s)ds\bigg)^{\frac{k}{k+1}}\bigg)\bigg(\frac{(v(x))^{\frac{k+1+\lambda}{k}}\tilde{L}(v(x))}{(k+1)\int_{0}^{v(x)}s^{\frac{1+\lambda}{k}}\tilde{L}(s)ds}\\
&-\frac{\lambda+1}{k}-\frac{v(x)\tilde{L}'(v(x))}{\tilde{L}(v(x))}\bigg)\bigg]\sum_{i=1}^{C_{N}^{k}}(-1)^{k}det(
v_{x_{is}x_{ij}}(x))(-\nabla v_{i}(x))^{T}B(v_{i}(x))\nabla
v_{i}(x)\bigg\}
\end{split}
\end{equation}
and
\[
\begin{split}
&\Phi\bigg(\tau_{j}\bigg(\int_{0}^{v(x)}s^{\frac{1+\lambda}{k}}\tilde{L}(s)ds\bigg)^{\frac{k}{k+1}}\bigg)\\
&=-\frac{\varphi'(\tau_{j}(\int_{0}^{v(x)}s^{\frac{1+\lambda}{k}}\tilde{L}(s)ds)^{\frac{k}{k+1}})}{\varphi''(\tau_{j}(\int_{0}^{v(x)}s^{\frac{1+\lambda}{k}}\tilde{L}(s)ds)^{\frac{k}{k+1}})\tau_{j}(\int_{0}^{v(x)}s^{\frac{1+\lambda}{k}}\tilde{L}(s)ds)^{\frac{k}{k+1}}},\,j=7,8.
\end{split}
\]
\end{theorem}
\subsection{The exact boundary behavior}
Let $\Lambda$ denote the set of all positive monotonic functions
$\theta\in C^{1}(0, \delta_{0})\cap L_{1}(0, \delta_{0})$ which
satisfy
\[
\lim_{t\rightarrow0^{+}}\frac{d}{dt}\bigg(\frac{\Theta(t)}{\theta(t)}\bigg):=D_{\theta}\in
[0, +\infty),\,\Theta(t)=\int_{0}^{t}f(s)ds.
\]
The set $\Lambda$ was first introduced by C\^{i}rstea and
R\u{a}dulescu \cite{HCirstea1}-\cite{HCirstea2} for non-decreasing
functions and by Mohammed \cite{HMohammedQY} for non-increasing
functions to study the boundary behavior and uniqueness of solutions
for boundary blow-up elliptic problems.

Our result can be summarized as follows.
\begin{theorem}\label{thm2.8}
Let $f$ satisfy $\mathbf{(S_{1})}$ $(\mbox{or } \mathbf{(S_{01})}
)$, $\mathbf{(f_{1})}$-$\mathbf{(f_{2})}$, $b$ satisfy
$\mathbf{(b_{1})}$ and the following condition
\begin{description}
\item[$\mathbf{(b_{5})}$] there exist constants $b_{1}<b_{2}$ and some
function $\theta\in\Lambda$ such that
\[
b_{1}:=\liminf_{d(x)\rightarrow0}\frac{b(x)}{\theta^{k+1}(d(x))}\leq
\limsup_{d(x)\rightarrow0}\frac{b(x)}{\theta^{k+1}(d(x))}=:b_{2}.
\]
\end{description}
If we further assume that
\[
C_{f}^{+\infty}>1
\]
or
\[
C_{f}^{+\infty}=1 \mbox{ and }D_{\theta}>0,
\]
then any classical $k$-convex solution $u$ to problem \eqref{M}
satisfies
\[
\tau_{10}^{1-C_{f}^{+\infty}}\leq\liminf_{d(x)\rightarrow0}\frac{u(x)}{\psi\big((\Theta(d(x)))^{\frac{k+1}{k}}\big)}\leq\limsup_{d(x)\rightarrow0}\frac{u(x)}{\psi\big((\Theta(d(x)))^{\frac{k+1}{k}}\big)}\leq\tau_{9}^{1-C_{f}^{+\infty}},
\]
where
\begin{equation}\label{3f27}
\tau_{9}=\frac{k}{k+1}\bigg(\frac{b_{1}k}{((k+1)(C_{f}^{+\infty}-1)+kD_{\theta})M_{k-1}^{+}}\bigg)^{1/k}
\end{equation}
and
\[
\tau_{10}=\frac{k}{k+1}\bigg(\frac{b_{2}k}{((k+1)(C_{f}^{+\infty}-1)+kD_{\theta})M_{k-1}^{-}}\bigg)^{1/k}
\]
with
\[
M_{k-1}^{+}=\max_{x\in\partial\Omega}S_{k-1}(\kappa_{1},\cdot\cdot\cdot,\kappa_{N-1})
\mbox{ and
}M_{k-1}^{-}=\min_{x\in\partial\Omega}S_{k-1}(\kappa_{1},\cdot\cdot\cdot,\kappa_{N-1}).
\]
In particular, if $C_{f}^{+\infty}=1$, then
\[
\lim_{d(x)\rightarrow0}\frac{u(x)}{\psi\big((\Theta(d(x)))^{\frac{k+1}{k}}\big)}=1.
\]
\end{theorem}
\section{Some preliminary results}
In this section, we collect some well-known results for the
convenience of later utilization and reference.
\begin{lemma}\label{lemma2.1}$($$\mbox{Lemma } 2.1 \mbox{ in } \cite{Prof.Jian})$
Suppose that $\Omega\subset \mathbb{R}^{N}$ is a bounded domain, and
$u,\,v\in C^{2}(\Omega)$ are $k$-convex. If
\begin{description}
\item[$\mathbf{(i)}$]$\phi_{1}(x, z, q)\geq \phi_{2}(x, z,
q),\,\forall\,(x, z, q)\in
\Omega\times\mathbb{R}\times\mathbb{R}^{N}$;
\item[$\mathbf{(ii)}$]$S_{k}(D^{2}u)\geq \phi_{1}(x, u, Du)$ and $S_{k}(D^{2}v)\leq \phi_{2}(x, v,
Dv)$ in $\Omega$;
\item[$\mathbf{(iii)}$] $u\leq v$ on $\partial\Omega$;
\item[$\mathbf{(iv)}$] $\partial_{z}\phi_{1}(x, z, p)>0$ or $\partial_{z}\phi_{2}(x, z,
p)>0$,
\end{description}
then $u\leq v$ in $\Omega$.
\end{lemma}
The following interior estimate for derivatives of smooth solutions
is a simple variant of Lemma 2.2 in \cite{HLazer-McKenna} and can be
proved by the idea of Theorem 3.1 and Remark 3.1 of \cite{WangXJIA}.
\begin{lemma}\label{lemma2.5}
Let $\Omega$ be a bounded $(k-1)$-convex domain in $\mathbb{R}^{N}$
with $N\geq 2$ and $\partial\Omega\in C^{\infty}$. Let $\beta\in
[-\infty, \infty)$ and $h\in C^{\infty}(\bar{\Omega}\times(\beta,
\infty))$ with $h(x, u)>0$ for $(x, u)\in \bar{\Omega}\times(\beta,
\infty)$. Let $u\in C^{\infty}(\bar{\Omega})$ be a solution of the
following problem
\[
S_{k}(D^{2}u)=h(x,
u),\,x\in\Omega,\,u|_{\partial\Omega}=c=\mbox{constant}
\]
with $\beta<u(x)<c$ in $\Omega$. Let $D\Subset\Omega$ be a subdomain
of $\Omega$ and assume that $\beta<\beta_{1}\leq u(x)\leq \beta_{2}$
for $x\in \bar{D}$ and let $\tau>1$ be an integer. Then there exists
a constant $C$ which depends only on $\beta_{1}, \beta_{2}$ and
$\tau$  bounds for the derivatives of $h(x, u)$ for $(x, u)\in
\bar{D}\times[\beta_{1}, \beta_{2}]$ and $\mbox{dist}(D,
\partial\Omega)$ such that
\[
||u||_{C^{\tau}(\bar{D})}\leq C.
\]
\end{lemma}
\begin{lemma}\label{lemma2.6} $(\mbox{Theorem }1.1 \mbox{ in } \cite{GuangBo}\mbox{ and }\cite{Salani})$
Let $\Omega$ be an open domain in $\mathbb{R}^{N}$ with
$C^{\infty}$-boundary and let $h(x, t)$ be a $C^{\infty}$-function
such that and $h>0$ and $h_{t}\geq0$ in $\Omega\times\mathbb{R}$.
Then the following problem
\[
S_{k}(D^{2}u)=h(x,
u),\,x\in\Omega,\,u|_{\partial\Omega}=\tilde{\phi}\in
C(\partial\Omega)
\]
has a unique $k$-convex solution provided  there exists a $k$-convex
strict subsolution $v$ in $\bar{\Omega}$, i.e., a $k$-convex
function $v$ such that $v|_{\partial\Omega}=\tilde{\phi}$ and
$S_{k}(D^{2}v)\geq h(x, v)+\delta$ in $\bar{\Omega}$, for some
$\delta>0$.
\end{lemma}
\begin{lemma}\label{lemma2.3}$($$\mbox{Proposition } 2.1 \mbox{ in } \cite{HuangYong}$ $ \mbox{ and  Lemma } 2.3 \mbox{ in } \cite{ZhangFeng1})$ Let $u\in C^{2}(\Omega)$ be such that all of the
principal submatrices of $(u_{x_{i}x_{j}})$ for $x\in \Omega$, and
let $h$ be a $C^{2}$-function defined on an interval containing the
range of $u$. Then
\[
S_{k}(D^{2}h(u))=S_{k}(D^{2}u)(h'(u))^{k}+(h'(u))^{k-1}h''(u)\sum_{i=1}^{C_{N}^{k}}\mbox{det}(u_{x_{is}x_{ij}})(\nabla
u_{i})^{T}B(u_{i})\nabla u_{i},
\]
 where $A^{T}$ denotes the transpose of the matrix $A$, $B(u_{i})$
denotes the inverse of the $i$-th principal submatrix
$(u_{x_{is}x_{ij}})$, $\mbox{det}(u_{x_{is}x_{ij}})$ denotes  the
determinant of $(u_{x_{is}x_{ij}})$ and
\[
\nabla u_{i}=(u_{x_{i1}}, \cdot\cdot\cdot,
u_{x_{ik}})^{T},\,i=1,\cdot\cdot\cdot, C_{N}^{k} \mbox{ and }
C_{N}^{k}=\frac{N!}{(N-k)!k!}.
\]
\end{lemma}
For any $\delta>0$, we define
\begin{equation}\label{cj}
\Omega_{\delta}:=\{x\in\Omega:0<d(x)<\delta\}.
\end{equation}
Since $\Omega$ is smooth, for positive integer $m\geq2$, we can
always take $\tilde{\delta}>0$ such that (please refer to Lemmas 14.
16 and 14. 17 in \cite{Gilbarg})
\begin{equation}\label{gtg}
d\in C^{m}(\Omega_{\tilde{\delta}}),\,|\nabla d(x)|=1,\,\forall\,
x\in\Omega_{\tilde{\delta}}.
\end{equation}
\begin{lemma}\label{lemma2.4}$(\mbox{Corollary } 2.3 \mbox{ in }\cite{HuangYong})$
Let $\Omega$ be bounded with $\partial\Omega\in C^{m}$ for $m\geq
2$. Assume that $x\in \Omega_{\delta_{1}}$ and $\bar{x}\in
\partial\Omega$  is the nearest point to $x$, i.e.,
$d(x)=|x-\bar{x}|$, then
\[
\begin{split}
S_{k}(D^{2}h(d(x)))&=(-h'(d(x)))^{k}S_{k}(\varepsilon_{1},
\cdot\cdot\cdot,
\varepsilon_{N-1})\\
&+(-h'(d(x)))^{k-1}h''(d(x))S_{k-1}(\varepsilon_{1},
\cdot\cdot\cdot, \varepsilon_{N-1}),
\end{split}
\]
where
\[
\varepsilon_{i}=\frac{\kappa_{i}(\bar{x})}{1-\kappa_{i}(\bar{x})d(x)},\,i=1,\cdot\cdot\cdot,
N-1
\]
and $\kappa_{i}(\bar{x})$ $i=1, \cdot\cdot\cdot, N-1$ are the
principal curvatures of $\partial\Omega$ at $\bar{x}$.
\end{lemma}

\section{Some basic facts from Karamata regular variation theory}
Some basic facts from Karamata regular variation theory are given in
this section, please refer to, for instance, \cite{BGT},
\cite{Maric}, \cite{Resnick} and Zhang's paper \cite{JFAZ}.
\begin{definition}\label{Def2.1} A positive continuous
function $f$ defined on $[a,+\infty)$, for some $a>0$, is called
{\bf regularly varying at positive infinity} with index $\rho$,
written $f \in RV_{\rho}$, if for each $\xi>0$ and some $\rho \in
\mathbb R$,
\begin{equation} \label{2.1}
\lim_{t\rightarrow +\infty} \frac{f(\xi t)}{f(t)}= \xi^\rho.
\end{equation}
In particular,  when $\rho=0$, $f$ is called   {\bf slowly varying
at infinity}. \noindent Clearly, if $f\in RV_\rho$, then
$\hat{L}(t):=f(t)/{t^\rho}$ is slowly varying at infinity.
\end{definition}
Similarly, we define the regularly varying functions at zero and at
infinity as follows.
\begin{definition}\label{definition4.2}
We also see that a positive continuous function $g$ defined on
$(0,a)$  for some $a>0$,  is {\bf regularly varying  at zero} with
index $\rho$ (write $g \in RVZ_\rho$) if $t\mapsto g(1/t)$ belongs
to $RV_{-\rho}$.
\end{definition}
\noindent We see from Definition \ref{definition4.2} that $g\in
RVZ_{\rho}$ $(\rho\in\mathbb{R})$, if for each $\xi>0$,
\begin{equation}\label{Definition0}
\lim_{t\rightarrow0^{+}}\frac{g(\xi t)}{g(t)}=\xi^{\rho}.
\end{equation}
\begin{definition}\label{definition4.3}
A positive continuous function $h$ defined on $(-\infty, a)$ for
some $a<0$, is {\bf regularly varying at negative infinity} with
index $\rho$ if $t\mapsto h(-t)$ is regularly varying at positive
infinity with index $\rho$.
\end{definition}
\begin{proposition}\label{ppp}  $(\mbox{Uniform Convergence
Theorem})$  If $f\in RV_\rho$, then  \eqref{2.1}  holds uniformly
for $\xi \in [c_1, c_2]$ with $0<c_1<c_2$. If $g\in RVZ_{\rho}$,
then \eqref{Definition0} holds uniformly for $\xi \in [c_1, c_2]$
with $0<c_1<c_2$.
\end{proposition}\label{Haan}
\begin{proposition} \label{2.2} $(\mbox{Representation Theorem})$
A function $L$ is slowly varying at positive infinity if and only if
it may be written in the form
\begin{equation*}\label{2.2}
L(t)=\nu(t) {\rm exp} \left( \int_{a_1}^t \frac {y(s)}{s} ds
\right), \
 t \geq a_1,
\end{equation*}
for some $a_1> 0$, where the functions $\nu$ and $y$ are continuous
and for $t \rightarrow +\infty$, $y(t)\rightarrow 0$ and
$\nu(t)\rightarrow c$, with $c>0$. We call that
\begin{equation*}\label{2.3}
 \hat{L}(t)=c {\rm exp} \left( \int_{a_1}^t
\frac {y(s)}{s} ds \right), \
 t \geq a_1,
\end{equation*}
  is {\bf normalized} slowly varying
 at positive infinity and
\begin{equation*}\label{2.4}
f(t)=t^\rho\hat{L}(t), \
 t \geq a_1
\end{equation*}
  is {\bf normalized} regularly varying at positive infinity with
 index $\rho$  $(\mbox{and write }f\in NRV_\rho).$
\end{proposition}
\begin{proposition}\label{proposition4.6}
A function $L$ is slowly varying at zero if and only if it may be
written in the form
\begin{equation*}
L(t)=\nu(t) {\rm exp} \left( \int^{a_1}_t \frac {y(s)}{s} ds
\right), \
 t \leq a_1,
\end{equation*}
for some $a_1> 0$, where the functions $\nu$ and $y$ are continuous
and for $t \rightarrow 0^{+}$, $y(t)\rightarrow 0$ and
$\nu(t)\rightarrow c$, with $c>0$. We call that
\begin{equation*}\label{2.3}
 \hat{L}(t)=c {\rm exp} \left( \int^{a_1}_t
\frac {y(s)}{s} ds \right), \
 t \leq a_1,
\end{equation*}
  is {\bf normalized} slowly varying
 at zero and
\begin{equation*}\label{2.4}
g(t)=t^\rho\hat{L}(t), \
 t \leq a_1
\end{equation*}
  is {\bf normalized} regularly varying at zero with
 index $\rho$  $(\mbox{and write } g\in NRVZ_\rho).$
\end{proposition}
\begin{proposition}\label{proposition4.7}
A function $L$ is slowly varying at negative infinity if and only if
it may be written in the form
\begin{equation*}
L(t)=\nu(t) {\rm exp} \left( \int^{a_1}_t \frac {y(s)}{s} ds
\right), \
 t \leq a_1,
\end{equation*}
for some $a_1< 0$, where the functions $\nu$ and $y$ are continuous
and for $t \rightarrow -\infty$, $y(t)\rightarrow 0$ and
$\nu(t)\rightarrow c$, with $c>0$. We call that
\begin{equation*}\label{2.3}
 \hat{L}(t)=c {\rm exp} \left( \int^{a_1}_t
\frac {y(s)}{s} ds \right), \
 t \leq a_1,
\end{equation*}
  is {\bf normalized} slowly varying
 at negative infinity and
\begin{equation*}\label{2.4}
h(t)=(-t)^\rho\hat{L}(t), \
 t \leq a_1
\end{equation*}
  is {\bf normalized} regularly varying at negative infinity with
 index $\rho$.
\end{proposition}
\begin{proposition}\label{proposition3.5}
A function $ f\in C^1[a_1, \infty)$ for some $a_{1}>0$ belongs to
$NRV_{\rho}$ if and only
 if
\begin{equation*}\label{smsp}
 \lim_{t \rightarrow +\infty}
 \frac{tf'(t)}{f(t)}=\rho.
\end{equation*}
A function $h\in C^1(0, a_{1}]$ for some $a_{1}>0$ belongs to
$NRVZ_{\rho}$ if and only
 if
\begin{equation*}\label{2.5}
 \lim_{t\rightarrow 0^{+}}
 \frac{th'(t)}{h(t)}=\rho.
\end{equation*}
A function $h\in C^1(-\infty, a_{1}]$ for some $a_{1}<0$ is
normalized regularly varying at negative infinity with index $\rho$
if and only
 if
\begin{equation*}\label{2.5}
 \lim_{t\rightarrow -\infty}
 \frac{th'(t)}{h(t)}=\rho.
\end{equation*}
\end{proposition}
\begin{proposition}\label{p3}
Let functions $L,\, L_{1}$ be slowly varying at zero and at negative
infinity, respectively, then
\begin{description}
\item[$\mathbf{(i)}$] for every $\rho > 0$ and $t\rightarrow 0^{+},\,$
$ t^{\rho}L(t)\rightarrow{0}, \quad
t^{-\rho}L(t)\rightarrow{\infty};$
\item[$\mathbf{(ii)}$] for $\rho>0$ and $t\rightarrow -\infty$,\, $(-t)^{-\rho}L_{1}(t) \rightarrow 0$
and $(-t)^{\rho}L_{1}(t)\rightarrow+\infty$.
\end{description}
\end{proposition}
\begin{proposition} \label{HUR}$(\mbox{Asymptotic Behavior}).$
Let $L$ and $L_{1}$ be slowly varying at positive infinity and at
zero, respectively and $a_{1}$ be a positive constant, then
\begin{description}
\item[$\mathbf{(i)}$]
$ \int_t^\infty s^{\rho}L(s)ds\sim (-\rho-1)^{-1}t^{1+\rho}L(t), \
t\rightarrow+\infty, \  for \ \rho< -1; $
\item[$\mathbf{(ii)}$]
$ \int_{t}^{a_{1}}s^{\rho}L_{1}(s)ds\sim
(-\rho-1)^{-1}t^{1+\rho}L_{1}(t),\,t\rightarrow0^{+}, \mbox{ for
}\rho<-1;$
\item[$\mathbf{(iii)}$]
$ \int_{a_{1}}^{t}s^{\rho}L(s)ds\sim
(\rho+1)^{-1}t^{1+\rho}L(t),\,t\rightarrow+\infty, \mbox{ for
}\rho>-1;$
\item[$\mathbf{(iv)}$]$\int_{0}^{t}s^{\rho}L_{1}(s)ds\sim (\rho+1)^{-1}t^{1+\rho}L_{1}(t),\,t\rightarrow0^{+}\mbox{ for
}\rho>-1$.
\end{description}
\end{proposition}
\section{Auxiliary results}
In this section, we collect some useful results, which are necessary
in the proofs of our results.
\begin{lemma}\label{lemma2j}$($Lemma $2.1$ in Lazer and McKenna $\cite{Lazer-McKenna1}$$)$
Let $g\in C^{1}(0, +\infty)$ is positive on $(0, +\infty)$,
$g'(t)\leq 0$, $\forall\,t\in(0, +\infty)$ and
$\lim_{t\rightarrow0^{+}}g(t)=+\infty$, then
\[
\lim_{t\rightarrow0^{+}}\frac{\int_{t}^{1}g(s)ds}{g(t)}=0.
\]
\end{lemma}
\begin{lemma}\label{lemma21}
Let $f$ satisfy $\mathbf{(S_{1})}$ $(\mbox{or }\mathbf{(S_{01})})$
and $\mathbf{(f_{1})}$, then
\begin{description}
\item[$\mathbf{(i)}$] if $\mathbf{(f_{2})}$ holds, then
$C_{f}^{+\infty}\geq1$ and $
\lim_{t\rightarrow+\infty}\frac{(f(t))^{1/k}}{t}\int_{t}^{+\infty}(f(s))^{-1/k}ds=C_{f}^{+\infty}-1;
$
\item[$\mathbf{(ii)}$] $\mathbf{(f_{2})}$ holds with
$C_{f}^{+\infty}>1$ if and only if $f\in NRV_{\gamma}$ with
$\gamma=\frac{kC_{f}^{+\infty}}{C_{f}^{+\infty}-1}$;
\item[$\mathbf{(iii)}$] if $\mathbf{(f_{2})}$ holds with
$C_{f}^{+\infty}=1$, then for any $\gamma>0$, it holds $\lim_{t\rightarrow+\infty}\frac{f(t)}{t^{\gamma}}=+\infty$. 
\end{description}
\end{lemma}
\begin{proof}
\noindent $\mathbf{(i)}$ We see
\begin{equation*}\label{2f3}
I(t)=((f(t))^{1/k})'\int_{t}^{+\infty}(f(s))^{-1/k}ds,\,t>0.
\end{equation*}
Integrate $I$ from $a>0$ to $t>a$ and  integrate by parts, we obtain
that
\[
\begin{split}
\int_{a}^{t}I(s)ds&=\int_{a}^{t}((f(s))^{1/k})'\int_{s}^{+\infty}(f(\tau))^{-1/k}d\tau
ds\\
&=(f(t))^{1/k}\int_{t}^{+\infty}(f(s))^{-1/k}ds-(f(a))^{1/k}\int_{a}^{+\infty}(f(s))^{-1/k}ds+t-a.
\end{split}
\]
It follows by the l'Hospital's rule that
\[
\lim_{t\rightarrow+\infty}\frac{(f(t))^{1/k}}{t}\int_{t}^{+\infty}(f(s))^{-1/k}ds=\lim_{t\rightarrow+\infty}I(t)-1=C_{f}^{+\infty}-1\geq0.
\]
So, we obtain $\mathbf{(i)}$ holds.

\noindent $\mathbf{(ii)}$ Necessity.  A straightforward calculation
shows that
\begin{equation}\label{2f1}
\lim_{t\rightarrow+\infty}\frac{f'(t)t}{f(t)}=\lim_{t\rightarrow+\infty}\frac{k((f(t))^{1/k})'\int_{t}^{+\infty}(f(s))^{-1/k}ds}{\frac{(f(t))^{1/k}}{t}\int_{t}^{+\infty}(f(s))^{-1/k}ds}=\frac{kC_{f}^{+\infty}}{C_{f}^{+\infty}-1}.
\end{equation}
It follows by Proposition \ref{proposition3.5} that $f\in
NRV_{\gamma}$ with
$\gamma=\frac{kC_{f}^{+\infty}}{C_{f}^{+\infty}-1}$.

\noindent  Sufficiency. 
 By Proposition \ref{2.2}, we see that there exist $a_{0}>0$ and
$L_{+\infty}\in NRV_{0}\cap C^{1}[a_{0}, +\infty)$ such that
\[
f(t)=t^{\gamma}L_{+\infty}(t),\,t\in[a_{0}, +\infty).
\]
By using Proposition \ref{HUR} $\mathbf{(i)}$ and a straightforward
calculation, we obtain
\[
\begin{split}
&\lim_{t\rightarrow+\infty}(t^{\gamma/k}(L_{+\infty}(t))^{1/k})'\int_{t}^{+\infty}s^{-\gamma/k}(L_{+\infty}(s))^{-1/k}ds\\
=&\lim_{t\rightarrow+\infty}\bigg(\frac{\gamma}{\gamma-k}+\frac{tL'_{+\infty}(t)}{(\gamma-k)L_{+\infty}(t)}\bigg)=\frac{\gamma}{\gamma-k}=C_{f}^{+\infty}.
\end{split}
\]
$\mathbf{(iii)}$ It follows by the similar calculation as
\eqref{2f1} that
\[
\lim_{t\rightarrow+\infty}\frac{tf'(t)}{f(t)}=+\infty.
\]
Therefore, for an arbitrary $\gamma>0$, there exists $t_{0}>0$ such
that
\[
\frac{f'(t)}{f(t)}>(\gamma+1)t^{-1},\,t\in[t_{0}, +\infty).
\]
Integrate it from $t_{0}$ to $t>t_{0}$, we obtain
\[
\ln f(t)-\ln f(t_{0})>(\gamma+1)(\ln t-\ln t_{0}),\,t>t_{0},
\]
i.e.,
\[
\frac{f(t)}{t^{\gamma}}>\frac{f(t_{0})}{t_{0}^{\gamma+1}}t,\,t>t_{0}.
\]
Letting $t\rightarrow +\infty$, we obtain $\mathbf{(iii)}$ holds.
\end{proof}
\begin{lemma}\label{lemma22}
Let $f$ satisfy $\mathbf{(S_{1})}$ and $\mathbf{(f_{1})}$, then
\begin{description}
\item[$\mathbf{(i)}$] if $\mathbf{(f_{3})}$ holds, then $C_{f}^{0}\geq
1$ and
$
\lim_{t\rightarrow0^{+}}\frac{(f(t))^{1/k}}{t}\int_{t}^{+\infty}(f(s))^{-1/k}ds=C_{f}^{0}-1;
$
\item[$\mathbf{(ii)}$] $\mathbf{(f_{3})}$ holds with $C_{f}^{0}>
1$ if and only if $f\in NRVZ_{\gamma}$ with
$\gamma=\frac{kC_{f}^{0}}{C_{f}^{0}-1}$;
\item[$\mathbf{(iii)}$] if $\mathbf{(f_{3})}$ holds with
$C_{f}^{0}=1$, then for any $\gamma>0$, it holds $\lim_{t\rightarrow0^{+}}\frac{f(t)}{t^{\gamma}}=0$.
\end{description}
\end{lemma}
\begin{proof}
\noindent $\mathbf{(i)}$ 
By Lemma \ref{lemma2j}, we see that
\[
\lim_{t\rightarrow0^{+}}(f(t))^{1/k}\int_{t}^{+\infty}(f(s))^{-1/k}ds=0.
\]
Integrate $I$ from $0$ to $t>0$ and integrate by parts, we obtain
that
\[
\int_{0}^{t}I(s)ds=\int_{0}^{t}((f(s))^{1/k})'\int_{s}^{+\infty}(f(\tau))^{-1/k}d\tau
ds=(f(t))^{1/k}\int_{t}^{+\infty}(f(s))^{-1/k}ds+t,
\]
It follows by the l'Hospital's rule that
\[
0\leq\lim_{t\rightarrow0^{+}}\frac{(f(t))^{1/k}}{t}\int_{t}^{+\infty}(f(s))^{-1/k}ds=\lim_{t\rightarrow0^{+}}I(t)-1=C_{f}^{0}-1.
\]
So, we obtain $\mathbf{(i)}$ holds.

\noindent $\mathbf{(ii)}$ Necessity. A straightforward calculation
shows that
\begin{equation}\label{2f2}
\lim_{t\rightarrow0^{+}}\frac{tf'(t)}{f(t)}=\lim_{t\rightarrow0^{+}}\frac{k((f(t))^{1/k})'\int_{t}^{+\infty}(f(s))^{-1/k}ds}{\frac{(f(t))^{1/k}}{t}\int_{t}^{+\infty}(f(s))^{-1/k}ds}=\frac{kC_{f}^{0}}{C_{f}^{0}-1}.
\end{equation}
It follows by Proposition \ref{proposition3.5} that $f\in
NRVZ_{\gamma}$ with $\gamma=\frac{kC_{f}^{0}}{C_{f}^{0}-1}$.

\noindent Sufficiency. 
By Proposition \ref{proposition4.6}, we see that there exist
$a_{0}>0$ and $L_{0}\in NRVZ_{0}\cap C^{1}(0, a_{0}]$ such that
\[
f(t)=t^{\gamma}L_{0}(t),\,t\in(0, a_{0}].
\]
A straightforward calculation shows that

\begin{equation}\label{2f9}
\begin{split}
&((f(t))^{1/k})'\int_{t}^{+\infty}(f(s))^{-1/k}ds\\
&=t^{\frac{\gamma-k}{k}}\bigg(\frac{\gamma}{k}(L_{0}(t))^{1/k}+\frac{1}{k}(L_{0}(t))^{1/k}\frac{tL_{0}'(t)}{L_{0}(t)}\bigg)\int_{t}^{a_{0}}s^{-\gamma/k}(L_{0}(s))^{-1/k}ds\\
&+t^{\frac{\gamma-k}{k}}\bigg(\frac{\gamma}{k}(L_{0}(t))^{1/k}+\frac{1}{k}(L_{0}(t))^{1/k}\frac{tL_{0}'(t)}{L_{0}(t)}\bigg)\int_{a_{0}}^{+\infty}(f(s))^{-1/k}ds.
\end{split}
\end{equation}
By $\mathbf{(f_{1})}$ and Proposition \ref{p3} $\mathbf{(i)}$, we
see that
\begin{equation}\label{2f10}
\lim_{t\rightarrow0^{+}}t^{\frac{\gamma-k}{k}}\bigg(\frac{\gamma}{k}(L_{0}(t))^{1/k}+\frac{1}{k}(L_{0}(t))^{1/k}\frac{tL_{0}'(t)}{L_{0}(t)}\bigg)\int_{a_{0}}^{+\infty}(f(s))^{-1/k}=0.
\end{equation}
\eqref{2f9}-\eqref{2f10} combined with Proposition \ref{HUR}
$\mathbf{(ii)}$ imply that
\[
\begin{split}
&\lim_{t\rightarrow0^{+}}(t^{\gamma/k}(L_{0}(t))^{1/k})'\int_{t}^{+\infty}s^{-\gamma/k}(L_{0}(s))^{1/k}ds\\
=&\lim_{t\rightarrow0^{+}}t^{\frac{\gamma-k}{k}}\bigg(\frac{\gamma}{k}(L_{0}(t))^{1/k}+\frac{1}{k}(L_{0}(t))^{1/k}\frac{tL_{0}'(t)}{L_{0}(t)}\bigg)\frac{k}{\gamma-k}t^{\frac{k-\gamma}{k}}(L_{0}(t))^{-1/k}
=\frac{\gamma}{\gamma-k}=C_{f}^{0}.
\end{split}
\]
$\mathbf{(iii)}$ It follows by the similar calculation as
\eqref{2f2} that
\[
\lim_{t\rightarrow0^{+}}\frac{tf'(t)}{f(t)}=+\infty.
\]
Therefore, for any $\gamma>0$, there exists $t_{0}>0$ such that
\[
\frac{f'(t)}{f(t)}>(\gamma+1)t^{-1},\,t\in(0, t_{0}].
\]
Integrate it from $t>0$ to $t_{0}$, we obtain
\[
\ln f(t_{0})-\ln f(t)>(\gamma+1)(\ln t_{0}-\ln t),\,t\in(0, t_{0}],
\]
i.e.,
\[
f(t)t^{-\gamma}<\frac{f(t_{0})t}{t_{0}^{\gamma+1}},\,t\in(0, t_{0}].
\]
Letting $t\rightarrow0$, we obtain $\mathbf{(iii)}$ holds.
\end{proof}
\begin{lemma}\label{lemma23}
Let $f$ satisfy $\mathbf{(S_{01})}$ and $\mathbf{(f_{1})}$, then
\begin{description}
\item[$\mathbf{(i)}$]if $\mathbf{(f_{4})}$ holds, then $C_{f}^{-\infty}\leq
1$ and
$
\lim_{t\rightarrow-\infty}\frac{(f(t))^{1/k}}{t}\int_{t}^{+\infty}(f(s))^{-1/k}ds=C_{f}^{-\infty}-1;
$
\item[$\mathbf{(ii)}$] $\mathbf{(f_{4})}$ holds with
$C_{f}^{-\infty}<1$ if and only if $f$ is normalized regularly
varying at negative infinity with index
$\frac{kC_{f}^{-\infty}}{C_{f}^{-\infty}-1}$;
\item[$\mathbf{(iii)}$] if $\mathbf{(f_{4})}$ holds with
$C_{f}^{-\infty}=1$, then for any $\gamma>0$, it holds
$\lim_{t\rightarrow-\infty}\frac{f(t)}{(-t)^{-\gamma}}=0$.
\end{description}
\end{lemma}
\begin{proof}
\noindent $\mathbf{(i)}$ Take $a\in\mathbb{R}$. Integrate $I$ from
$t$ to $a>t$ and integrate by parts, we obtain
\[
\begin{split}
\int_{t}^{a}I(s)ds&=\int_{t}^{a}((f(s))^{1/k})'\int_{s}^{+\infty}(f(\tau))^{-1/k}d\tau
ds\\
&=(f(a))^{1/k}\int_{a}^{+\infty}(f(s))^{-1/k}ds-(f(t))^{1/k}\int_{t}^{+\infty}(f(s))^{-1/k}ds+a-t.
\end{split}
\]
It follows by the l'Hospital's rule that
\[
0\geq\lim_{t\rightarrow-\infty}\frac{(f(t))^{1/k}}{t}\int_{t}^{+\infty}(f(s))^{-1/k}ds=-1+\lim_{t\rightarrow-\infty}I(t)=-1+C_{f}^{-\infty}.
\]
So, we obtain $\mathbf{(i)}$ holds.

\noindent $\mathbf{(ii)}$ Necessity.  A straightforward calculation
shows that
\begin{equation}\label{2f4}
\lim_{t\rightarrow-\infty}\frac{tf'(t)}{f(t)}=\frac{k((f(t))^{1/k})'\int_{t}^{+\infty}(f(s))^{-1/k}ds}{\frac{(f(t))^{1/k}}{t}\int_{t}^{+\infty}(f(s))^{-1/k}ds}=\frac{kC_{f}^{-\infty}}{C_{f}^{-\infty}-1}.
\end{equation}
It follows by Proposition \ref{proposition3.5} that $f\in
NRVZ_{\gamma}$ with
$\gamma=\frac{kC_{f}^{-\infty}}{C_{f}^{-\infty}-1}$.

\noindent Sufficiency. By Proposition \ref{proposition4.7}, we see
that there exist $a_{0}<0$ and a slowly varying function at negative
infinity $L_{-\infty}\in C^{1}(-\infty, a_{0}]$ such that
\[
f(t)=(-t)^{\gamma}L_{-\infty}(t),\,t\leq a_{0}.
\]
A straightforward calculation shows that
\begin{equation}\label{2f7}
\begin{split}
&((f(t))^{1/k})'\int_{t}^{+\infty}(f(s))^{-1/k}ds\\
&=(-t)^{\frac{\gamma-k}{k}}\bigg(-\frac{\gamma}{k}(L_{-\infty}(t))^{1/k}-\frac{1}{k}(L_{-\infty}(t))^{1/k}\frac{tL'_{-\infty}(t)}{L_{-\infty}(t)}\bigg)\int_{t}^{a_{0}}(-s)^{-\gamma/k}(L_{-\infty}(s))^{-1/k}ds\\
&+(-t)^{\frac{\gamma-k}{k}}\bigg(-\frac{\gamma}{k}(L_{-\infty}(t))^{1/k}-\frac{1}{k}(L_{-\infty}(t))^{1/k}\frac{tL'_{-\infty}(t)}{L_{-\infty}(t)}\bigg)\int_{a_{0}}^{+\infty}(f(s))^{-1/k}ds.
\end{split}
\end{equation}
By $\mathbf{(f_{1})}$ and Proposition \ref{p3} $\mathbf{(ii)}$, we
have
\begin{equation}\label{2f8}
\lim_{t\rightarrow-\infty}(-t)^{\frac{\gamma-k}{k}}\bigg(-\frac{\gamma}{k}(L_{-\infty}(t))^{1/k}-\frac{1}{k}(L_{-\infty}(t))^{1/k}\frac{tL'_{-\infty}(t)}{L_{-\infty}(t)}\bigg)\int_{a_{0}}^{+\infty}(f(s))^{-1/k}ds=0.
\end{equation}
Moreover, by using Proposition \ref{HUR} $\mathbf{(iii)}$, we have
\begin{equation*}
\int_{t}^{a_{0}}(-s)^{-\gamma/k}(L_{-\infty}(s))^{-1/k}ds\sim\frac{k}{k-\gamma}(-t)^{\frac{k-\gamma}{k}}(L_{-\infty}(t))^{-1/k},\,t\rightarrow-\infty.
\end{equation*}
This combined with \eqref{2f7}-\eqref{2f8}  implies that
\[
\begin{split}
&\lim_{t\rightarrow-\infty}((f(t))^{1/k})'\int_{t}^{+\infty}(f(s))^{-1/k}ds\\
=&\lim_{t\rightarrow-\infty}\bigg(\frac{\gamma}{\gamma-k}+\frac{tL_{-\infty}'(t)}{(\gamma-k)L_{-\infty}(t)}\bigg)=\frac{\gamma}{\gamma-k}=C_{f}^{-\infty}.
\end{split}
\]
\noindent $\mathbf{(iii)}$ It follows by the similar calculation as
\eqref{2f4} that
\[
\lim_{t\rightarrow-\infty}\frac{tf'(t)}{f(t)}=-\infty.
\]
Therefore, for any $\gamma>0$, there exists $t_{0}<0$ such that
\[
\frac{f'(t)}{f(t)}>-(\gamma+1)t^{-1},\,t\in(-\infty, t_{0}].
\]
Integrate it from $t$ to $t_{0}$, we obtain
\[
\ln f(t_{0})-\ln f(t)>(\gamma+1)(\ln t-\ln t_{0}),\,t\in(-\infty,
t_{0}],
\]
i.e.,
\[
(-t)^{\gamma}f(t)<f(t_{0})(-t_{0})^{\gamma+1}(-t)^{-1},\,t\in(-\infty,
t_{0}].
\]
Letting $t\rightarrow-\infty$, we obtain $\mathbf{(iii)}$ holds.
\end{proof}
\begin{lemma}\label{lemma24}
Let $f$ satisfy $\mathbf{(S_{1})}$ $($or $\mathbf{(S_{01})}$$)$ and
$\mathbf{(f_{1})}$, $\psi$ be uniquely determined by \eqref{f22},
then
\begin{description}
\item[$\mathbf{(i)}$] $\psi'(t)=-(f(\psi(t)))^{1/k}$ \mbox{ and
}$\psi''(t)=\frac{1}{k}(f(\psi(t)))^{\frac{2-k}{k}}f'(\psi(t)),\,t>0$;
\item[$\mathbf{(ii)}$] if $\mathbf{(f_{2})}$ holds, then
$
\lim_{t\rightarrow0^{+}}\frac{t\psi'(t)}{\psi(t)}=1-C_{f}^{+\infty}
\mbox{ and
}\lim_{t\rightarrow0^{+}}\frac{t\psi''(t)}{\psi'(t)}=-C_{f}^{+\infty};
$
\item[$\mathbf{(iii)}$]if $\mathbf{(S_{1})}$ and $\mathbf{(f_{3})}$
hold, then
$
\lim_{t\rightarrow+\infty}\frac{t\psi'(t)}{\psi(t)}=1-C_{f}^{0}
\mbox{ and }
\lim_{t\rightarrow+\infty}\frac{t\psi''(t)}{\psi'(t)}=-C_{f}^{0};
$
\item[$\mathbf{(iv)}$] if $\mathbf{(S_{01})}$ and $\mathbf{(f_{4})}$
hold, then
$
\lim_{t\rightarrow+\infty}\frac{t\psi'(t)}{\psi(t)}=1-C_{f}^{-\infty}
\mbox{ and
}\lim_{t\rightarrow+\infty}\frac{t\psi''(t)}{\psi'(t)}=-C_{f}^{-\infty}.
$
\end{description}
\end{lemma}
\begin{proof}
\noindent $\mathbf{(i)}$ By a direct calculation, we obtain
$\mathbf{(i)}$ holds.

\noindent $\mathbf{(ii)}$ By $\mathbf{(i)}$ and Lemma \ref{lemma21}
$\mathbf{(i)}$, we have
\[
\begin{split}
\lim_{t\rightarrow0^{+}}\frac{t\psi'(t)}{\psi(t)}&=-\lim_{t\rightarrow0^{+}}\frac{(f(\psi(t)))^{1/k}}{\psi(t)}\int_{\psi(t)}^{+\infty}(f(s))^{-1/k}ds\\
&=-\lim_{t\rightarrow+\infty}\frac{(f(t))^{1/k}}{t}\int_{t}^{+\infty}(f(s))^{-1/k}ds=1-C_{f}^{+\infty}.
\end{split}
\]
Moreover, by $\mathbf{(i)}$ and $\mathbf{(f_{2})}$, we have
\[
\begin{split}
\lim_{t\rightarrow0^{+}}\frac{t\psi''(t)}{\psi'(t)}&=-\lim_{t\rightarrow0^{+}}\frac{1}{k}(f(\psi(t)))^{\frac{1-k}{k}}f'(\psi(t))\int_{\psi(t)}^{+\infty}(f(s))^{-1/k}ds\\
&=-\lim_{t\rightarrow+\infty}((f(t))^{1/k})'\int_{t}^{+\infty}(f(s))^{-1/k}ds=-C_{f}^{+\infty}.
\end{split}
\]
$\mathbf{(iii)}$ By $\mathbf{(i)}$ and Lemma \ref{lemma22}
$\mathbf{(i)}$, we have
\[
\begin{split}
\lim_{t\rightarrow+\infty}\frac{t\psi'(t)}{\psi(t)}&=-\lim_{t\rightarrow+\infty}\frac{(f(\psi(t)))^{1/k}}{\psi(t)}\int_{\psi(t)}^{+\infty}(f(s))^{-1/k}ds\\
&=-\lim_{t\rightarrow0^{+}}\frac{(f(t))^{1/k}}{t}\int_{t}^{+\infty}(f(s))^{-1/k}ds=1-C_{f}^{0}.
\end{split}
\]
Moreover, by $\mathbf{(i)}$ and $\mathbf{(f_{3})}$, we have
\[
\begin{split}
\lim_{t\rightarrow+\infty}\frac{t\psi''(t)}{\psi'(t)}&=-\lim_{t\rightarrow+\infty}\frac{1}{k}(f(\psi(t)))^{\frac{1-k}{k}}f'(\psi(t))\int_{\psi(t)}^{+\infty}(f(s))^{-1/k}ds\\
&=-\lim_{t\rightarrow0^{+}}((f(t))^{1/k})'\int_{t}^{+\infty}(f(s))^{-1/k}ds=-C_{f}^{0}.
\end{split}
\]
$\mathbf{(iv)}$ By $\mathbf{(i)}$ and Lemma \ref{lemma23}
$\mathbf{(i)}$, we have
\[
\begin{split}
\lim_{t\rightarrow+\infty}\frac{t\psi'(t)}{\psi(t)}&=-\lim_{t\rightarrow+\infty}\frac{(f(\psi(t)))^{1/k}}{\psi(t)}\int_{\psi(t)}^{+\infty}(f(s))^{-1/k}ds\\
&=-\lim_{t\rightarrow
-\infty}\frac{(f(t))^{1/k}}{t}\int_{t}^{+\infty}(f(s))^{-1/k}ds=1-C_{f}^{-\infty}.
\end{split}
\]
Moreover, by $\mathbf{(i)}$ and $\mathbf{(f_{4})}$, we have
\[
\begin{split}
\lim_{t\rightarrow+\infty}\frac{t\psi''(t)}{\psi'(t)}&=-\lim_{t\rightarrow+\infty}\frac{1}{k}(f(\psi(t)))^{\frac{1-k}{k}}f'(\psi(t))\int_{\psi(t)}^{+\infty}(f(s))^{-1/k}ds\\
&=-\lim_{t\rightarrow-\infty}((f(t))^{1/k})'\int_{t}^{+\infty}(f(s))^{-1/k}ds=-C_{f}^{-\infty}.
\end{split}
\]
\end{proof}
\begin{lemma}\label{lemma25}
Let $f$ satisfy $\mathbf{(S_{1})}$ $(\mbox{or }\mathbf{(S_{01})})$
and $\mathbf{(f_{5})}$, $F$ be defined by \eqref{f21},  then
\begin{description}
\item[$\mathbf{(i)}$]$\lim_{t\rightarrow+\infty}\frac{(F(t))^{1/k}}{t}=+\infty$;
\item[$\mathbf{(ii)}$] if $f\in RV_{k}$, then
\[
\lim_{t\rightarrow+\infty}\frac{(F(t))^{1/(k+1)}}{t}\int_{t}^{+\infty}(F(s))^{-1/(k+1)}ds=+\infty
\]
and
\[
\lim_{t\rightarrow+\infty}((F(t))^{1/(k+1)})'\int_{t}^{+\infty}(F(s))^{-1/(k+1)}ds=+\infty.
\]
\end{description}
\end{lemma}
\begin{proof}
\noindent $\mathbf{(i)}$ If $\mathbf{(i)}$ is false, then there
exist constant $c_{*}$ and an increasing sequence of real numbers
$\{s_{i}\}_{i=1}^{+\infty}$ satisfying
$\lim_{i\rightarrow+\infty}s_{i}=+\infty$ and $2s_{i-1}\leq
s_{i},\,i=1,2,\cdot\cdot\cdot$ such that
\[
(F(s_{i}))^{-1/k}\geq 1/(s_{i}c_{*}).
\]
A direct calculation shows that
\[
\begin{split}
+\infty>\int_{s_{0}}^{+\infty}(F(s))^{-1/k}ds
&\geq \sum_{i=1}^{+\infty}\int_{i-1}^{i}(F(s))^{-1/k}ds\\
&\geq \sum_{i=1}^{+\infty}\int_{i-1}^{i}(F(s_{i}))^{-1/k}ds\\
&\geq
\sum_{i=1}^{+\infty}\int_{s_{i-1}}^{s_{i}}1/(s_{i}c_{*})ds=\sum_{i=1}^{+\infty}c_{*}^{-1}\bigg(\frac{s_{i}-s_{i-1}}{s_{i}}\bigg)\geq
\lim_{i\rightarrow+\infty}\frac{i}{2c_{*}}=+\infty.
\end{split}
\]
This is a contradiction. So, $\mathbf{(i)}$ holds.

\noindent $\mathbf{(ii)}$ 
Since
\[
\lim_{t\rightarrow+\infty}\frac{\int_{\upsilon}^{0}f(s)ds}{tf(t)}=0,
\]
by the Lebesgue's dominated convergence theorem, we obtain
\begin{equation}\label{2f5}
\begin{split}
\lim_{t\rightarrow+\infty}\frac{F(t)}{tf(t)}&=\lim_{t\rightarrow+\infty}\bigg(\frac{\int_{\upsilon}^{0}f(s)ds}{tf(t)}+\frac{\int_{0}^{t}f(s)ds}{tf(t)}\bigg)\\
&=\lim_{t\rightarrow+\infty}\frac{\int_{0}^{t}f(s)ds}{tf(t)}=\lim_{t\rightarrow+\infty}\int_{0}^{1}\frac{f(t\tau)}{f(t)}d\tau=\int_{0}^{1}\tau^{k}d\tau=\frac{1}{k+1}.
\end{split}
\end{equation}
It follows by using the l'Hospital's rule that
\begin{equation}\label{2f6}
\begin{split}
&\lim_{t\rightarrow+\infty}\frac{(F(t))^{1/(k+1)}}{t}\int_{t}^{+\infty}(F(s))^{-1/(k+1)}ds\\
=&\lim_{t\rightarrow+\infty}\frac{\int_{t}^{+\infty}(F(s))^{-1/(k+1)}ds}{t(F(t))^{-1/(k+1)}}=\lim_{t\rightarrow+\infty}\bigg(\frac{1}{k+1}\frac{tf(t)}{F(t)}-1\bigg)^{-1}=+\infty.
\end{split}
\end{equation}
Combining \eqref{2f5} and \eqref{2f6}, we have
\[
\begin{split}
&\lim_{t\rightarrow+\infty}((F(t))^{1/(k+1)})'\int_{t}^{+\infty}(F(s))^{-1/(k+1)}ds\\
=&\lim_{t\rightarrow+\infty}\frac{\frac{1}{k+1}\frac{(F(t))^{1/(k+1)}}{t}\int_{t}^{+\infty}(F(s))^{-1/(k+1)}ds}{\frac{F(t)}{tf(t)}}=+\infty.
\end{split}
\]
\end{proof}
\begin{lemma}\label{lemma26}
Let $f$ satisfy $\mathbf{(S_{1})}$ and $\mathbf{(f_{5})}$, $F$ be
defined by \eqref{f21}, then
\begin{description}
\item[$\mathbf{(i)}$]if $\mathbf{(f_{7})}$ holds, then
$E_{f}^{0}\geq1$ and $ \lim_{t\rightarrow0^{+}}
\frac{(F(t))^{1/(k+1)}}{t}\int_{t}^{+\infty}(F(s))^{-1/(k+1)}ds=E_{f}^{0}-1;
$
\item[$\mathbf{(ii)}$] if $\mathbf{(f_{7})}$ holds with
$E_{f}^{0}>1$ if and only if $f\in RVZ_{\gamma}$ with
$\gamma=\frac{kE_{f}^{0}+1}{E_{f}^{0}-1}$;
\item[$\mathbf{(iii)}$] if $\mathbf{(f_{7})}$ holds with
$E_{f}^{0}=1$, then for any $\gamma>0$, it holds
$\lim_{t\rightarrow0^{+}}\frac{F(t)}{t^{\gamma}}=0$.
\end{description}
\end{lemma}
\begin{proof}
\noindent $\mathbf{(i)}$ We see
\[
J(t)=((F(t))^{1/(k+1)})'\int_{t}^{+\infty}(F(s))^{-1/(k+1)}ds,\,t>0.
\]
By Lemma \ref{lemma2j}, we see that
\[
\lim_{t\rightarrow0^{+}} (F(t))^{1/(k+1)}\int_{t}^{+\infty}(F(s))^{-1/(k+1)}ds=0.
\]
Integrate $J$  from $0$ to $t>0$ and integrate by parts, we obtain
\[
\int_{0}^{t}J(s)ds=(F(t))^{1/(k+1)}\int_{t}^{+\infty} (F(s))^{-1/(k+1)} ds+t,\,t>0.
\]
It follows by the l'Hospital's rule that
\[
0\leq\lim_{t\rightarrow0^{+}} \frac{(F(t))^{1/(k+1)}}{t}
\int_{t}^{+\infty}(F(s))^{-1/(k+1)}ds=\lim_{t\rightarrow0^{+}}J(t)-1=E_{f}^{0}-1.
\]
So, we obtain $\mathbf{(i)}$ holds.

\noindent $\mathbf{(ii)}$ Necessity. A straightforward calculation
shows that
\[
\lim_{t\rightarrow0^{+}}\frac{F(t)}{tf(t)}=\frac{1}{k+1}\frac{(F(t))^{1/(k+1)}\int_{t}^{+\infty}(F(s))^{-1/(k+1)}ds}{tJ(t)}
=\frac{E_{f}^{0}-1}{(k+1)E_{f}^{0}}.
\]
It follows by Proposition \ref{proposition3.5} that $F\in
NRVZ_{\frac{(k+1)E_{f}^{0}}{E_{f}^{0}-1}}$. This implies that there
exist $a_{1}>0$ and $\hat{L}_{0}\in NRVZ_{0}\cap C^{1}(0, a_{1}]$
such that
\[
F(t)=t^{\frac{(k+1)E_{f}^{0}}{E_{f}^{0}-1}}\hat{L}_{0}(t),\,t\in(0,
a_{1}].
\]
By a simple calculation, we obtain
\[
f(t)=t^{\gamma}\bigg(\frac{(k+1)E_{f}^{0}}{E_{f}^{0}-1}
+\frac{t\hat{L}'_{0}(t)}{\hat{L}_{0}(t)}\bigg)\hat{L}_{0}(t) \mbox{
with }\gamma=\frac{kE_{f}^{0}+1}{E_{f}^{0}-1}.
\]
It follows by Proposition \ref{proposition4.6} that $f\in
NRV_{\gamma}$.

\noindent Sufficiency. Since $f\in RVZ_{\gamma}$, by the Lebesgue's
dominated convergence
 theorem, we obtain
\[
\lim_{t\rightarrow0^{+}}\frac{F(t)}{tf(t)}=\lim_{t\rightarrow0^{+}}\frac{\int_{0}^{t}f(s)ds}{tf(t)}=\lim_{t\rightarrow0^{+}}\int_{0}^{1}\frac{f(t\tau)}{f(t)}d\tau=\int_{0}^{1}\tau^{\gamma}d\tau=\frac{1}{\gamma+1}.
\]
So, $F\in NRVZ_{\gamma+1}$. By Proposition \ref{proposition4.6}, we
see that there exist $a_{1}>0$  and $\hat{L}_{0}\in C^{1}(0, a_{1}]$
such that
\[
F(t)=t^{\gamma+1}\hat{L}_{0}(t),\,t\in(0, a_{1}].
\]
A straightforward calculation shows that
\begin{equation}\label{2f11}
\begin{split}
&((F(t))^{1/(k+1)})'\int_{t}^{+\infty}(F(s))^{-1/(k+1)}ds\\
&=t^{\frac{\gamma-k}{k+1}}(\hat{L}_{0}(t))^{\frac{1}{k+1}}\bigg(\frac{\gamma+1}{k+1}+\frac{1}{k+1}\frac{t\hat{L}'_{0}(t)}{\hat{L}_{0}(t)}\bigg)\int_{t}^{a_{1}}t^{-(\gamma+1)/(k+1)}(\hat{L}_{0}(s))^{-1/(k+1)}ds\\
&+t^{\frac{\gamma-k}{k+1}}(\hat{L}_{0}(t))^{\frac{1}{k+1}}\bigg(\frac{\gamma+1}{k+1}+\frac{1}{k+1}\frac{t\hat{L}'_{0}(t)}{\hat{L}_{0}(t)}\bigg)\int_{a_{1}}^{+\infty}(F(s))^{-1/(k+1)}ds.
\end{split}
\end{equation}
By $\mathbf{(f_{5})}$ and Proposition \ref{p3} $\mathbf{(i)}$, we
have
\begin{equation}\label{2f12}
\lim_{t\rightarrow0^{+}}t^{\frac{\gamma-k}{k+1}}(\hat{L}_{0}(t))^{\frac{1}{k+1}}\bigg(\frac{\gamma+1}{k+1}+\frac{1}{k+1}\frac{t\hat{L}'_{0}(t)}{\hat{L}_{0}(t)}\bigg)\int_{a_{1}}^{+\infty}(F(s))^{-1/(k+1)}ds=0.
\end{equation}
\eqref{2f11}-\eqref{2f12} combined with Proposition \ref{HUR}
$\mathbf{(ii)}$ imply that
\[
\begin{split}
&\lim_{t\rightarrow0^{+}}((F(t))^{1/(k+1)})'\int_{t}^{+\infty}(F(s))^{-1/(k+1)}ds\\
=&\lim_{t\rightarrow0^{+}}\bigg(\frac{\gamma+1}{\gamma-k}+\frac{1}{\gamma-k}\frac{t\hat{L}_{0}'(t)}{\hat{L}_{0}(t)}\bigg)=\frac{\gamma+1}{\gamma-k}=E_{f}^{0}.
\end{split}
\]
$\mathbf{(iii)}$  By the similar argument as Lemma \ref{lemma22}
$\mathbf{(iii)}$, we obtain $\mathbf{(iii)}$ holds.

\end{proof}
\begin{lemma}\label{lemma27}
Let $f$ satisfy $\mathbf{(S_{01})}$ and $\mathbf{(f_{5})}$, $F$ be
defined by \eqref{f21}, then
\begin{description}
\item[$\mathbf{(i)}$]if $\mathbf{(f_{8})}$ holds, then
$E_{f}^{-\infty}\leq1$ and $
\lim_{t\rightarrow-\infty}\frac{(F(t))^{1/(k+1)}}{t}\int_{t}^{+\infty}(F(s))^{-1/(k+1)}ds=E_{f}^{-\infty}-1;
$
\item[$\mathbf{(ii)}$]if $\mathbf{(f_{8})}$ holds with
$E_{f}^{-\infty}<1$ if and only if $F$ is regularly varying at
negative infinity with index
$\frac{E_{f}^{-\infty}(k+1)}{E_{f}^{-\infty}-1}$;
\item[$\mathbf{(iii)}$] if $\mathbf{(f_{8})}$ holds with
$E_{f}^{-\infty}=1$, then for any $\gamma>0$, it holds
$\lim_{t\rightarrow-\infty}\frac{F(t)}{(-t)^{-\gamma}}=0$.
\end{description}
\end{lemma}
\begin{proof}
Take $a\in\mathbb{R}$. Integrate $J$ from $t$ to $a>t$ and integrate by parts, we obtain
\[
\begin{split}
\int_{t}^{a}J(s)ds&=\int_{t}^{a}((F(s))^{1/(k+1)})'\int_{a}^{+\infty}(F(s))^{-1/(k+1)}ds\\
&=(F(a))^{1/(k+1)}\int_{a}^{+\infty}(F(s))^{-1/(k+1)}ds-(F(t))^{1/(k+1)}\int_{t}^{+\infty}(F(s))^{-1/(k+1)}ds+a-t.
\end{split}
\]
It follows by the l'Hospital's rule that
\[
0\geq
\lim_{t\rightarrow-\infty}\frac{(F(t))^{1/(k+1)}}{t}\int_{t}^{+\infty}(F(s))^{-1/(k+1)}ds=-1+\lim_{t\rightarrow-\infty}J(t)=E_{f}^{-\infty}-1.
\]
$\mathbf{(ii)}$-$\mathbf{(iii)}$ By similar arguments as
$\mathbf{(ii)}$ and $\mathbf{(iii)}$ of Lemma \ref{lemma26}, we
obtain $\mathbf{(ii)}$-$\mathbf{(iii)}$ hold.
\end{proof}
\begin{lemma}\label{lemma29}
Let $f$ satisfy $\mathbf{(S_{1})}$ $(\mbox{or } \mathbf{S_{01}})$
and $\mathbf{(f_{5})}$, $\varphi$  be uniquely determined by
\eqref{f23}, then
\begin{description}
\item[$\mathbf{(i)}$] $\varphi'(t)=-((k+1)F(\varphi(t)))^{1/(k+1)}$,
$\psi''(t)=((k+1)F(\varphi(t)))^{\frac{1-k}{k+1}}f(\varphi(t))$ and
\[
(-\varphi'(t))^{k-1}\varphi''(t)=f(\varphi(t)),\,t>0;
\]
\item[$\mathbf{(ii)}$]if $\mathbf{(f_{6})}$ holds, then $\lim_{t\rightarrow0^{+}}\frac{\varphi'(t)}{t\varphi''(t)}=0;$
\item[$\mathbf{(iii)}$] if $\mathbf{(f_{7})}$ holds, then $\lim_{t\rightarrow+\infty}\frac{\varphi'(t)}{t\varphi''(t)}=-(E_{f}^{0})^{-1};$
\item[$\mathbf{(iv)}$]  if $\mathbf{(f_{8})}$ holds, then $\lim_{t\rightarrow+\infty}\frac{\varphi'(t)}{t\varphi''(t)}=-(E_{f}^{-\infty})^{-1}.$
\end{description}
\end{lemma}
\begin{proof}
\noindent By a direct calculation, we obtain $\mathbf{(i)}$ holds.

\noindent $\mathbf{(ii)}$ By $\mathbf{(i)}$ and $\mathbf{(f_{6})}$, we have
\[
\begin{split}
\lim_{t\rightarrow0^{+}}
\frac{\varphi'(t)}{t\varphi''(t)}=-&\lim_{t\rightarrow0^{+}}\bigg(\frac{1}{k+1}(F(\varphi(t)))^{-k/(k+1)}f(\varphi(t))\int_{t}^{+\infty}(F(s))^{-1/(k+1)}ds\bigg)^{-1}\\
=-&\lim_{t\rightarrow+\infty}\bigg(((F(t))^{1/(k+1)})^{'}\int_{t}^{+\infty}(F(s))^{-1/(k+1)}ds\bigg)^{-1}=0.
\end{split}
\]
\noindent $\mathbf{(iii)}$ By $\mathbf{(i)}$ and $\mathbf{(f_{7})}$, we have
\[
\lim_{t\rightarrow+\infty}\frac{\varphi'(t)}{t\varphi''(t)}=-\lim_{t\rightarrow0^{+}}\bigg(((F(t))^{1/(k+1)})'\int_{t}^{+\infty}(F(s))^{-1/(k+1)}ds\bigg)^{-1}=-(E_{f}^{0})^{-1}.
\]
$\mathbf{(iv)}$  By $\mathbf{(i)}$  and $\mathbf{(f_{8})}$, we have
\[
\lim_{t\rightarrow+\infty}\frac{\varphi'(t)}{t\varphi''(t)}=-\lim_{t\rightarrow-\infty}\bigg(((F(t))^{1/(k+1)})'\int_{t}^{+\infty}(F(s))^{-1/(k+1)}ds\bigg)^{-1}=-(E_{f}^{-\infty})^{-1}.
\]
\end{proof}
\begin{lemma}\label{lemma28}$(\mbox{Lemma } 2.1 \mbox{ in }\cite{DSZhang1})$
Let $\theta\in \Lambda$, we have
\begin{description}
\item[$\mathbf{(i)}$]$\lim_{t\rightarrow0^{+}}\frac{\Theta(t)}{\theta(t)}=0 \mbox{ and }
\lim_{t\rightarrow0^{+}}\frac{\Theta(t)\theta'(t)}{\theta^{2}(t)}=1-D_{\theta}$;
\item[$\mathbf{(ii)}$]if $D_{\theta}>0$, then $\theta\in
NRVZ_{\frac{1-D_{\theta}}{D_{\theta}}}$, in particular, if
$D_{\theta}=1$, then $\theta$ is slowly varying at zero;
\item[$\mathbf{(iii)}$]if $\theta=0$, then for any $\gamma>0$, it
holds $\lim_{t\rightarrow0^{+}}\frac{\theta(t)}{t^{\gamma}}=0$.
\end{description}
\end{lemma}
\section{Optimal global behavior of large solutions}
In this section, we prove Theorems \ref{thm2.1}-\ref{thm2.6}. Basic
to our subsequent discussions is the following two lemmas.

We first introduce a sub-supersolution method of $k$-convex solution
to  problem \eqref{M}.

\begin{definition}
A function $\underline{u}\in C^{2}(\Omega)$ is called subsolution if
$\underline{u}$ is $k$-convex in $\Omega$ and satisfies
\[
S_{k}(D^{2}\underline{u}(x))\geq
b(x)f(\underline{u}(x)),\,x\in\Omega,\,\underline{u}|_{\partial\Omega}=+\infty.
\]
\end{definition}
\begin{definition}
A function $\underline{u}\in C^{2}(\Omega)$ is called supersolution
if $\overline{u}$ is $k$-convex in $\Omega$ and satisfies
\[
S_{k}(D^{2}\overline{u}(x))\leq
b(x)f(\overline{u}(x)),\,x\in\Omega,\,\overline{u}|_{\partial\Omega}=+\infty.
\]
\end{definition}
\begin{lemma}\label{Lemma4.1}
Let $b$ satisfy $\mathbf{(b_{1})}$, $f$ satisfy $\mathbf{(S_{1})}$
$(\mbox{or } \mathbf{(S_{01})})$. Assume that
$\overline{u}=h_{1}(v)$ and $\underline{u}=h_{2}(v)$ are positive
$($or may be sign-changing$)$ classical supersolution and
subsolution, respectively to problem \eqref{M} and satisfy
$\underline{u}\leq \overline{u}$ in $\Omega$, then problem \eqref{M}
has at least one $k$-convex solution $u\in C^{\infty}(\Omega)$  in
the order interval $[\underline{u}, \overline{u}]$.
\end{lemma}
Inspired by the ideas of Lazer and Mckenna in Theorem 2.1 of
\cite{HLazer-McKenna} and Zhang and Feng  in Theorem 1.3 of
\cite{ZhangFeng1}, we prove Lemma \ref{Lemma4.1}.
\begin{proof}
Since $\underline{u}=h_{2}(v)$ is a subsolution to problem
\eqref{M}, we have
\[
S_{k}(D^{2}\underline{u}(x))\geq
b(x)f(\underline{u}(x)),\,x\in\Omega.
\]
Take $\varepsilon>0$ and let
$w_{\varepsilon}(x):=\underline{u}(x)-\varepsilon$, it is clear that
$w_{\varepsilon}(x)$ is also a subsolution to problem \eqref{M}. In
fact, if $\mathbf{(S_{1})}$ holds, we can take $\varepsilon>0$ small
enough such that $w_{\varepsilon}$ is positive in $\Omega$. So, we
have
\begin{equation}\label{4f4}
S_{k}(D^{2}w_{\varepsilon}(x))>b(x)f(w_{\varepsilon}(x)),\,x\in\Omega.
\end{equation}

Let $\{\sigma_{n}\}_{n=1}^{+\infty}$ be a strictly increasing
sequence of positive numbers such that
$\sigma_{n}\rightarrow+\infty$ as $n\rightarrow+\infty$, and let
\[
\Omega_{n}^{\varepsilon}:=\{x\in\Omega:
w_{\varepsilon}(x)<\sigma_{n}\} \mbox{ and
}\Omega_{n}:=\{x\in\Omega:\underline{u}(x)<\sigma_{n}\}.
\]
Since any level surface of $\underline{u}$ is a level surface of
$v$, for each $n\geq1$, $\partial\Omega_{n}^{\varepsilon}$ and
$\partial\Omega_{n}$ are strictly convex $C^{\infty}$-submanifold of
$\mathbb{R}^{N}$ of dimension $N-1$.

When $\mathbf{(S_{1})}$ holds,  we  take $\varepsilon_{0}>0$ such
that $w_{\varepsilon_{0}}$ is positive in $\Omega$. Define
\[
a_{*}:=\min_{x\in\Omega}w_{\varepsilon_{0}}(x).
\]
For this case, we extend $f$ from $f\in C^{\infty}[a_{*},+\infty)$
to $\tilde{f}\in C^{\infty}(\mathbb{R})$, where $\tilde{f}$ is
increasing on $\mathbb{R}$ and $\tilde{f}=f$ on $[a_{*}, +\infty)$.

Next, we still denote  $\tilde{f}$ by $f$ for convenience.

Take $\varepsilon<\varepsilon_{0}$, it is clear that
$w_{\varepsilon}>w_{\varepsilon_{0}} \mbox{ in }\Omega.$ We see from
\eqref{4f4} that
\[
S_{k}(D^{2}w_{\varepsilon}(x))>b(x)f(w_{\varepsilon}(x)),\,x\in\bar{\Omega}^{\varepsilon}_{n}.
\] By Lemma
\ref{lemma2.6},  there exists a $k$-convex function
$u_{n}^{\varepsilon}$ for $n\geq 1$ such that
\begin{equation}\label{3f31}
S_{k}(D^{2}u_{n}^{\varepsilon}(x))=b(x)f(u_{n}^{\varepsilon}(x)),\,x\in
\Omega_{n}^{\varepsilon},\,u_{n}^{\varepsilon}|_{\partial\Omega_{n}^{\varepsilon}}=\sigma_{n}.
\end{equation}
We conclude from Lemma \ref{lemma2.1} that
\begin{equation}\label{3f28}
w_{\varepsilon}\leq u_{n}^{\varepsilon} \mbox{ in
}\bar{\Omega}_{n}^{\varepsilon},\,\,\,w_{\varepsilon}\leq
u_{n+1}^{\varepsilon} \mbox{ in }\bar{\Omega}_{n+1}^{\varepsilon}
\end{equation}
and
\[
u_{n}^{\varepsilon}\leq \overline{u}\mbox{ in
}\Omega_{n}^{\varepsilon}.
\]
Moreover, by the definitions of $\Omega_{n}^{\varepsilon}$ and
$\Omega_{n}$, we see that for any $\varepsilon>0$, the following
hold
\begin{equation}\label{3f29}
\bar{\Omega}_{n}^{\varepsilon}\subset\Omega_{n+1}^{\varepsilon},\,\bar{\Omega}_{n}\subset
\Omega_{n+1},\,\Omega_{n}\subset \Omega_{n}^{\varepsilon} \mbox{ and
}\bigcup_{n=1}^{\infty}\Omega_{n}^{\varepsilon}=\bigcup_{n=1}^{\infty}\Omega_{n}=\Omega.
\end{equation}
Combining \eqref{3f28} and \eqref{3f29}, we obtain
$u_{n+1}^{\varepsilon}\geq w_{\varepsilon}=u_{n}^{\varepsilon}$ on
$\partial\Omega_{n}^{\varepsilon}$. It follows from Lemma
\ref{lemma2.1} that
\[
u_{n}^{\varepsilon}\leq u_{n+1}^{\varepsilon}  \mbox{ in
}\bar{\Omega}_{n}^{\varepsilon}.
\]
So, we have
\begin{equation}\label{3f30}
w_{\varepsilon}\leq u_{n}^{\varepsilon}\leq
u_{n+1}^{\varepsilon}\leq \overline{u} \mbox{ in }\Omega_{n}.
\end{equation}

Let $\{\varepsilon_{n}\}_{n=1}^{+\infty}$ with
$\varepsilon_{1}<\varepsilon_{0}$ be a strictly decreasing sequence
of positive numbers such that $\varepsilon_{n}\rightarrow0$ as
$n\rightarrow+\infty$. By \eqref{3f30}, we have for fixed
$\varepsilon_{i}$,
\begin{equation}\label{3f32}
w_{\varepsilon_{i}}\leq u_{n}^{\varepsilon_{i}}\leq
u_{n+1}^{\varepsilon_{i}}\leq \overline{u} \mbox{ in }\Omega_{n}.
\end{equation}
This implies that for every $x\in \Omega$,
\[
u^{i}(x):=\lim_{n\rightarrow+\infty}u_{n}^{\varepsilon_{i}}(x)
\mbox{ exists}
\]
and
\[
\min_{x\in\bar{\Omega}_{n}}w_{\varepsilon_{1}}(x)\leq u^{i}(x)\leq
\max_{x\in\bar{\Omega}_{n}}\overline{u}(x),\,x\in \bar{\Omega}_{n}.
\]
Fix an integer $m$. For any $n>m$, we have
\[
\bar{\Omega}_{m}\subset\Omega_{n}
\]
and
\[
\mbox{dist}(\bar{\Omega}_{m}, \partial\Omega_{m+1})\leq
\mbox{dist}(\bar{\Omega}_{m}, \partial\Omega_{n})\leq
\mbox{dist}(\Omega_{m}, \partial\Omega).
\]
Since $u_{n}^{\varepsilon_{i}}$ is a $k$-convex solution to
\eqref{3f31}, by Lemma \ref{lemma2.5}, we obtain that for fixed
integer $j\geq 3$, there exists a positive constant $C_{j, m}$
(corresponding to $j$ and $m$) independent of $n$ such that for any
$n\geq m$, it holds
\[
||u_{n}^{\varepsilon_{i}}||_{C^{j}(\bar{\Omega}_{m})}\leq C_{j, m}.
\]
By Arzela-Ascoli's theorem, we can take a subsequence of
$\{u_{n}^{\varepsilon_{i}}\}_{n=1}^{+\infty}$, still denoted by
itself, such that $u_{n}^{\varepsilon_{i}}\rightarrow u^{i}$ in
$C^{j-1}(\bar{\Omega}_{m})$. Hence, for any $x\in \bar{\Omega}_{m}$,
the following holds
\[
S_{k}(D^{2}u^{i}(x))=\lim_{n\rightarrow\infty}S_{k}(D^{2}u_{n}^{\varepsilon_{i}}(x))=b(x)\lim_{n\rightarrow\infty}f(u_{n}^{\varepsilon_{i}}(x))=b(x)f(u^{i}(x)).
\]
Since $j, m$ are arbitrary, $u_{n}^{\varepsilon_{i}}$ is $k$-convex
solution of \eqref{3f31} and
\[
\lim_{d(x)\rightarrow0}w_{\varepsilon_{1}}(x)=+\infty,
\]
we obtain that $u^{i}\in C^{\infty}(\Omega)$ is a $k$-convex
solution to problem \eqref{M}.

On the other hand, we note that
\[
u_{n}^{\varepsilon_{i}}(x)\leq u_{n}^{\varepsilon_{i+1}}(x)\leq
\overline{u}(x),\,x\in \Omega_{n}.
\]
This combined with \eqref{3f32} shows that
\[
w_{\varepsilon_{i}}(x)\leq u^{i}(x)\leq u^{i+1}(x)\leq
\overline{u}(x),\,x\in\Omega.
\]
So, we have
\[
u(x):=\lim_{i\rightarrow+\infty}u^{i}(x) \mbox{ exists} \mbox{ and }
w_{\varepsilon_{i}}(x)\leq u(x)\leq \overline{u}(x),\,x\in\Omega.
\]
Passing to $i\rightarrow+\infty$, we obtain
\[
\underline{u}(x)\leq u(x)\leq \overline{u}(x),\,x\in\Omega.
\]
By the same argument as the  above, we  see that $u\in
C^{\infty}(\Omega)$ is $k$-convex and satisfies \eqref{M}.
\end{proof}
\begin{lemma}\label{Lemma4.2}
Let $\mathcal {X}$ be an arbitrary interval and  $h(x,t)$ be
 a continuous function on $\bar{\Omega}\times\mathcal {X}$,  then
\[t\mapsto\max_{x\in\bar{\Omega}}h(x, t) \mbox{ and }
t\mapsto\min_{x\in\bar{\Omega}}h(x, t)\] are continuous on $\mathcal
{X}$.
\end{lemma}
\begin{proof}
$\forall\,t_{0}\in\mathcal {X}$, we show that
$t\mapsto\max_{x\in\bar{\Omega}}h(x, t)$ is continuous at $t_{0}$.
Otherwise, there exist constant $\varepsilon_{0}>0$ and a sequence
  of numbers $\{t_{n}\}_{n=1}^{+\infty}$ satisfying
$t_{n}\rightarrow t_{0}$ as $n\rightarrow+\infty$ such that
\begin{equation}\label{3f33}
\big|\max_{x\in \bar{\Omega}}h(x, t_{n})-\max_{x\in\bar{\Omega}}h(x,
t_{0})\big|>\varepsilon_{0}.
\end{equation}
Since $h(x, t_{0})$ is continuous in $\bar{\Omega}$, we can take
$x_{0}\in\bar{\Omega}$ such that
\begin{equation}\label{3f34}
\max_{x\in\bar{\Omega}}h(x, t_{0})=h(x_{0}, t_{0}).
\end{equation}
In same way, we can take $x_{n}\in \bar{\Omega}$ such that
\begin{equation}\label{3f35}
\max_{x\in\bar{\Omega}} h(x, t_{n})=h(x_{n}, t_{n}).
\end{equation}
It follows by \eqref{3f33}-\eqref{3f35} that
\begin{equation}\label{3f36}
\big|h(x_{n}, t_{n})-h(x_{0}, t_{0})\big|>\varepsilon_{0}.
\end{equation}
Since $\bar{\Omega}$ is a bounded closed domain in $\mathbb{R}^{N}$,
$\{x_{n}\}_{n=1}^{+\infty}$ has a convergent subsequence. For
convenience, we still denote the subsequence by
$\{x_{n}\}_{n=1}^{+\infty}$. So, there exists $x_{*}\in
\bar{\Omega}$ such that $(x_{n}, t_{n})\rightarrow (x_{*}, t_{0})$.
This together with \eqref{3f36} implies that
\begin{equation}\label{3f37}
\big|h(x_{*}, t_{0})-h(x_{0}, t_{0})\big|\geq \varepsilon_{0}.
\end{equation}

Recalling \eqref{3f35}, we see that
\[
h(x_{n}, t_{n})\geq h(x_{0}, t_{n}).
\]
Letting $n\rightarrow+\infty$, we obtain
\[
h(x_{*}, t_{0})\geq h(x_{0}, t_{0}).
\]
This together with \eqref{3f37} implies that
\[
\max_{x\in\bar{\Omega}}h(x, t_{0})<h(x_{*}, t_{0}).
\]
This is a contradiction with \eqref{3f34}.

By similar arguments as the above, we can obtain
$t\mapsto\min_{x\in\bar{\Omega}}h(x, t)$ is continuous on $\mathcal
{X}$.
\end{proof}
\begin{remark}
In Lemma \ref{Lemma4.2}, if we replace the interval $[t_{*}, t^{*}]$
by $\mathbb{R}$, then the result still  holds.
\end{remark}
\subsection{Proof of Theorem \ref{thm2.1}}
\begin{proof}
The definition of $v$ implies that $(v_{x_{i}x_{j}})$ is negative
definite on $\bar{\Omega}$. So, each principal submatrix of size $k$
of $(v_{x_{i}x_{j}})$ is also negative definite. It follows that
there exist positive constants $e_{1}<e_{2}$ such that
\begin{equation}\label{f6}
e_{1}||\nabla v_{i}(x)||^{2}\leq(-\nabla
v_{i}(x))^{T}B(v_{i}(x))\nabla v_{i}(x)\leq e_{2}||\nabla
v_{i}(x)||^{2} \mbox{ and }\Delta
v_{i}=\sum_{j=1}^{k}v_{x_{ij}x_{ij}}(x)<0,\,x\in\bar{\Omega}.
\end{equation}
From Hopf's maximum principle, there exist  positive constants $e$
and $\delta_{1}$ such that
\begin{equation}\label{f7}
||\nabla v_{i}||^{2}>e \mbox{ in }\bar{\Omega}_{\delta_{1}}.
\end{equation}
where $\Omega_{\delta_{1}}$ is defined as shown in \eqref{cj}.
Combining \eqref{f6}-\eqref{f7}, we obtain that
\begin{equation}\label{3f14}
\sum_{i=1}^{C_{N}^{k}}(-1)^{k}det(v_{x_{is}x_{ij}})(-\nabla
v_{i})^{T}B(v_{i})\nabla v_{i}
\end{equation}
is nonnegative in $\bar{\Omega}$ and positive in
$\bar{\Omega}_{\delta_{1}}$. So, we have
\[
\omega_{0}>0 \mbox{ in }\bar{\Omega},
\]
where $\omega_{0}$ is given by \eqref{fb}.

Let $\bar{u}_{\lambda}(x)=m_{0}(v(x))^{-\alpha},\,x\in\Omega$, where
$m_{0}$ and $\alpha$ are given by \eqref{fa}. By
$\mathbf{(b_{1})}$-$\mathbf{(b_{2})}$ and Lemma \ref{lemma2.3}, we
obtain
\begin{equation}\label{3f8}
\begin{split}
S_{k}(D^{2}\underline{u}_{\lambda}(x))&=(m_{0}\alpha)^{k}(v(x))^{-(\alpha+1)k-1}\omega_{0}(x)\\
&\geq
c_{0}(m_{0}\alpha)^{k}(v(x))^{-(\alpha+1)k-1}=c_{0}\alpha^{k}m_{0}^{k-\gamma}(v(x))^{\lambda}(m_{0}(v(x))^{-\alpha})^{\gamma}\geq
b(x)\underline{u}_{\lambda}^{\gamma}(x),\,x\in\Omega,
\end{split}
\end{equation}
i.e., $\underline{u}_{\lambda}$ is a subsolution to problem
\eqref{M} in $\Omega$. Moreover, by a similar calculation as
\eqref{3f8}, we see that $S_{i}(D^{2}\underline{u}_{\lambda})>0$ in
$\Omega$ for $i=1,\cdot\cdot\cdot,N$. This implies that
$\underline{u}_{\lambda}$ is strictly convex in $\Omega$.

In a similar way, we can show that
$\overline{u}_{\lambda}=M_{0}v^{-\alpha}$ is a strictly convex
supersolution to problem \eqref{M} in $\Omega$, where $M$ is given
by \eqref{fa}. Since $\underline{u}_{\lambda}\leq
\overline{u}_{\lambda}$ in $\Omega$, by Theorem 4.2 of \cite{Salani}
and Lemma \ref{Lemma4.1}, we obtain that problem \eqref{M} has a
unique classical $k$-convex solution $u_{\lambda}\in
C^{\infty}(\Omega)$ satisfying
\begin{equation}\label{3f4}
\underline{u}_{\lambda}(x)\leq u_{\lambda}(x)\leq
\overline{u}_{\lambda}(x),\,x\in\Omega.
\end{equation}
We see by Lemma \ref{Lemma4.2} that $c_{0}$ and $C_{0}$ are
continuous on $[-k-1, +\infty)$ and $c_{0}(-k-1)=c_{1}$ and
$C_{0}(-k-1)=C_{1}$,  where $c_{1}$ and $C_{1}$ are given by
\eqref{f1}. Let $\Omega_{1}$ be an arbitrary compact subset of
$\Omega$. Passing to $\lambda\rightarrow-k-1$ and
$\lambda\rightarrow +\infty$, we obtain that \eqref{f8}-\eqref{f9}
hold.
\end{proof}
\subsection{Proof of Theorem
\ref{thm2.2}}
\begin{proof}
By the similar argument as in the proof of Theorem \ref{thm2.1}, we
see that
\[
\omega_{1}>0 \mbox{ in }\bar{\Omega},
\]
where $\omega_{1}$ is given by \eqref{2J1}.

Let $\underline{u}_{\lambda}(x)=m_{1}-(k+1+\lambda)\ln
v(x),\,x\in\Omega$, where $m_{1}$ is given by \eqref{f13}. By
$\mathbf{(b_{1})}$-$\mathbf{(b_{2})}$ and Lemma \ref{lemma2.3}, we
obtain
\begin{equation}\label{3f9}
\begin{split}
S_{k}(D^{2}\underline{u}_{\lambda}(x))&=(k+1+\lambda)^{k}(v(x))^{\lambda}(v(x))^{-(k+1+\lambda)}\omega_{1}(x)\\
&\geq
c_{1}(k+1+\lambda)^{k}\exp(-m_{1})(v(x))^{\lambda}\exp(\underline{u}_{\lambda}(x))\geq
b(x)\exp(\underline{u}_{\lambda}(x)),\,x\in\Omega,
\end{split}
\end{equation}
i.e., $\underline{u}_{\lambda}$ is a subsolution to problem
\eqref{M} in $\Omega$. Moreover, by a similar calculation as
\eqref{3f9}, we see that $\underline{u}_{\lambda}$ is  strictly
convex in $\Omega$.

In a similar way, we can show that
$\overline{u}_{\lambda}=M_{1}-(k+1+\lambda)\ln v$ is a strictly
convex supersolution in $\Omega$, where $M_{1}$ is given by
\eqref{f13}. Obviously, $\underline{u}_{\lambda}\leq
\overline{u}_{\lambda}$ in $\Omega$.  So, by Lemma \ref{Lemma4.1},
we obtain that problem \eqref{M} has a classical $k$-convex solution
$u_{\lambda}\in C^{\infty}(\Omega)$ satisfying \eqref{3f4}. By a
direct calculation, we see that \eqref{f14}-\eqref{f15} hold.
\end{proof}
\subsection{Proof of Theorem \ref{thm2.3}}
\begin{proof}
If $\mathbf{(f_{1})}$-$\mathbf{(f_{3})}$ hold, then  by Lemma
\ref{lemma24} $\mathbf{(ii)}$-$\mathbf{(iii)}$, we see that
\begin{equation}\label{3f1}
0<h_{0}=\inf\limits_{t>0}\Psi(t)\leq  \min\{C_{f}^{0},\,
C_{f}^{+\infty}\}\leq\max\{C_{f}^{0},\,
C_{f}^{+\infty}\}\leq\sup\limits_{t>0}\Psi(t):=H_{0}<+\infty.
\end{equation}
If $\mathbf{(f_{1})}$-$\mathbf{(f_{2})}$ and $\mathbf{(f_{4})}$
hold, then by Lemma \ref{lemma24} $\mathbf{(ii)}$ and
$\mathbf{(iv)}$, we obtain
\begin{equation}\label{3f2}
0<h_{0}=\inf\limits_{t>0}\Psi(t)\leq C_{f}^{-\infty}\leq1\leq
C_{f}^{+\infty}\leq\sup\limits_{t>0}\Psi(t):=H_{0}<+\infty.
\end{equation}
Moreover, by a direct calculation, we have
\begin{equation}\label{3f3}
\Psi(t)=\frac{1}{k}f^{\frac{1-k}{k}}(\psi(t))f'(\psi(t))t.
\end{equation}
Combining \eqref{3f1}-\eqref{3f3}, \eqref{f11} and a similar
argument as in the proof of Theorem \ref{thm2.1}, we obtain that
there exist positive constants $\beta_{1}<\beta_{2}$ such that for
any $\tau>0$, it holds
\begin{equation}\label{3f38}
\beta_{1}<\omega_{2}(\tau, x)<\beta_{2},\,\forall\,
x\in\bar{\Omega},
\end{equation}
where $\omega_{2}$ is given by \eqref{2J2}.

On the other hand, we see from Lemma \ref{Lemma4.2} that
\[
\tau\mapsto\tau^{k}\max_{x\in\bar{\Omega}}\omega_{2}(\tau, x) \mbox{
and } \tau\mapsto\tau^{k}\min_{x\in\bar{\Omega}}\omega_{2}(\tau, x)
\]
are continuous functions on $(0, +\infty)$. This together with
\eqref{3f38} implies that
\[
\tau^{k}\beta_{1}\leq\tau^{k}\min_{x\in\bar{\Omega}}\omega_{2}(\tau,
x)\leq\tau^{k}\max_{x\in\bar{\Omega}}\omega_{2}(\tau,
x)\leq\tau^{k}\beta_{2}.
\]
The existence theorem of zero point of continuous function implies
that there exist positive constants $\tau_{1}$ and $\tau_{2}$
$(\tau_{1}\leq\tau_{2})$ such that \eqref{f12} holds.

Let
$\underline{u}_{\lambda}(x)=\psi(\tau_{2}\eta^{-1}(v(x))^{\eta}),\,x\in\Omega$,
where $\eta$ is given by \eqref{ftj}. By
$\mathbf{(b_{1})}$-$\mathbf{(b_{2})}$ and Lemma \ref{lemma2.3}, we
obtain
\begin{equation}\label{3f7}
\begin{split}
S_{k}(D^{2}\underline{u}_{\sigma}(x))&=\tau^{k}(-\psi'(\tau_{2}\eta^{-1}(v(x))^{\eta}))^{k}(v(x))^{(\eta-1)k-1}\omega_{2}(\tau_{2}, x)\\
&\geq \tau_{2}^{k}\min_{x\in\bar{\Omega}}\omega(\tau_{2}, x)(v(x))^{\lambda}f(\psi(\tau_{2}\eta^{-1}(v(x))^{\eta}))\\
&=b_{2}(v(x))^{\lambda}f(\psi(\tau_{2}\eta^{-1}(v(x))^{\eta}))=b(x)f(\psi(\tau_{2}\eta^{-1}(v(x))^{\eta})),\,x\in\Omega,
\end{split}
\end{equation}
i.e., $\underline{u}_{\lambda}$ is a subsolution to problem
\eqref{M} in $\Omega$. Moreover, by a similar calculation as
\eqref{3f7}, we see that $\underline{u}_{\lambda}$ is strictly
convex in $\Omega$.

In a similar way, we can show that
$\overline{u}_{\lambda}=\psi(\tau_{1}\eta^{-1}v^{\eta})$ is a
strictly convex supersolution to problem \eqref{M} in $\Omega$.
Obviously, $\underline{u}_{\lambda}\leq \overline{u}_{\lambda}$ in
$\Omega$. So, by Lemma \ref{Lemma4.1}, we obtain that problem
\eqref{M} has a classical $k$-convex solution $u_{\lambda}\in
C^{\infty}(\Omega)$ satisfying \eqref{3f4}.

Take $\eta_{*}>0$ and let
\[
\begin{split}
\overline{\omega}(x,
\eta)&=v(x)(-1)^{k}S_{k}(D^{2}v(x))\\
&+((H_{0}-1)\eta+1)\sum_{i=1}^{C_{N}^{k}}(-1)^{k}det(v_{x_{is}x_{ij}}(x))(-\nabla
v_{i}(x))B(v_{i}(x))\nabla v_{i}(x),\,(x,
\eta)\in\bar{\Omega}\times[0, \eta_{*}]
\end{split}
\]
and
\[
\begin{split}
\underline{\omega}(x,
\eta)&=v(x)(-1)^{k}S_{k}(D^{2}v(x))\\
&+((h_{0}-1)\eta+1)\sum_{i=1}^{C_{N}^{k}}(-1)^{k}det(v_{x_{is}x_{ij}}(x))(-\nabla
v_{i}(x))B(v_{i}(x))\nabla v_{i}(x),\,(x,
\eta)\in\bar{\Omega}\times[0, \eta_{*}],
\end{split}
\]
So,  we have
\[
\bigg(\frac{b_{1}}{\max\limits_{(x, \eta)\in\bar{\Omega}\times[0,
\eta_{*}]}\overline{\omega}(x,
\eta)}\bigg)^{1/k}\leq\tau_{1}\leq\tau_{2}\leq\bigg(\frac{b_{2}}{\min\limits_{(x,
\eta)\in\bar{\Omega}\times[0, \eta_{*}]}\underline{\omega}(x,
\eta)}\bigg)^{1/k}.
\]
If $\mathbf{(f_{3})}$ holds, then by Lemma \ref{lemma24}
$\mathbf{(iii)}$ and Proposition \ref{ppp}, we see that
\begin{equation}\label{3f5}
\lim_{\eta\rightarrow0^{+}}\frac{\psi(\tau_{j}\eta^{-1}(v(x))^{\eta})}{\psi(\eta^{-1})}=\tau_{j}^{1-C_{f}^{0}},\,j=1,2.
\end{equation}
If $\mathbf{(f_{4})}$ holds, then by Lemma \ref{lemma24}
$\mathbf{(iv)}$ and Proposition \ref{ppp}, we see that
\begin{equation}\label{3f6}
\lim_{\eta\rightarrow0^{+}}\frac{\psi(\tau_{j}\eta^{-1}(v(x))^{\eta})}{\psi(\eta^{-1})}=\tau_{j}^{1-C_{f}^{-\infty}},\,j=1,2.
\end{equation}
Thus, \eqref{3f5} (or \eqref{3f6}) implies that \eqref{f2} (or
\eqref{f3}) holds. On the other hand, it follows by
\[
\psi((b_{2}/c_{1})^{1/k}\eta^{-1})\leq
\min_{x\in\bar{\Omega}_{1}}\psi(\tau_{2}\eta^{-1}(v(x))^{\eta})
\mbox{ and }
\lim_{\eta\rightarrow+\infty}\psi((b_{2}/c_{1})^{1/k}\eta^{-1})=+\infty,
\]
that
\[
\lim_{\eta\rightarrow+\infty}\min_{x\in\bar{\Omega}_{1}}u_{\lambda}(x)=+\infty.
\]
\end{proof}
\subsection{Proof of Theorem \ref{thm2.4}}
\begin{proof}
Without loss of generality, we assume
\[
\max_{x\in\bar{\Omega}}v(x)<\exp(-\mu).
\]
This implies that $1-\mu(\ln v(x))^{-1}>0,\,x\in\bar{\Omega}$. By
\eqref{3f1}-\eqref{3f3} and a similar argument as in the proof of
Theorem \ref{thm2.1}, we obtain that there exist positive constant
$\beta_{3}<\beta_{4}$ such that for any $\tau>0$, it holds
\[
\beta_{3}<\omega_{3}(\tau, x)<\beta_{4},\,x\in\bar{\Omega},
\]
where $\omega_{3}$ is given by \eqref{2J3}.

By the similar argument as in the proof of Theorem \ref{thm2.3}, we
see that there exist positive constants $\tau_{3}$ and $\tau_{4}$
$(\tau_{3}\leq \tau_{4})$ such that \eqref{f16} holds.

Let $\underline{u}_{\mu}=\psi(\tau_{4}(\mu-1)^{-1}(-\ln
v(x))^{1-\mu}),\,x\in\Omega$. By $\mathbf{(b_{1})}$,
$\mathbf{(b_{3})}$ and Lemma \ref{lemma2.3}, we obtain
\begin{equation}\label{3f10}
\begin{split}
S_{k}(D^{2}\underline{u}_{\mu}(x))&=\tau_{4}^{k}(-\psi'(\tau_{4}(\mu-1)^{-1}(-\ln
v(x))^{1-\mu}))^{k}(v(x))^{-(k+1)}(-\ln v(x))^{-\mu
k}\omega_{3}(\tau_{4}, x)\\
&\geq b_{2}(v(x))^{-(k+1)}(-\ln v(x))^{\mu
k}f(\underline{u}_{\mu})\geq
b(x)f(\underline{u}_{\mu}),\,x\in\Omega,
\end{split}
\end{equation}
i.e., $\underline{u}_{\mu}$ is a subsolution to problem \eqref{M} in
$\Omega$. Moreover, by a similar calculation as \eqref{3f10}, we see
that $\underline{u}_{\mu}$ is strictly convex in $\Omega$.

In a similar way, we can show that $\overline{u}_{\mu}$ is a
strictly convex supersolution to problem \eqref{M} in $\Omega$.
Obviously, $\underline{u}_{\mu}\leq \overline{u}_{\mu}$ in $\Omega$.
So, by Lemma \ref{Lemma4.1}, we obtain that problem \eqref{M} has a
classical $k$-convex solution $u_{\mu}\in C^{\infty}(\Omega)$
satisfying
\[
\underline{u}_{\mu}(x)\leq u_{\mu}(x)\leq
\overline{u}_{\mu}(x),\,x\in\Omega.
\]

Take $\mu_{*}>1$ and we may as well assume that $\max_{x\in
\bar{\Omega}}v(x)<\exp(-\mu_{*})$. Define
\[
\begin{split}
\overline{\omega}_{*}(x,\mu)&=v(x)(-1)^{k}S_{k}(D^{2}v(x))+\big(H_{0}(\mu-1)(-\ln v(x))^{-1}\\
&+1-\mu(\ln
v(x))^{-1}\big)\sum_{i=1}^{C_{N}^{k}}(-1)^{k}\mbox{det}(v_{x_{is}x_{ij}}(x))\\
&\times(-\nabla v_{i}(x))^{T}B(v_{i}(x))\nabla v_{i}(x),\,(x,
\mu)\in \bar{\Omega}\times [1, \mu_{*}]
\end{split}
\]
and
\[
\begin{split}
\underline{\omega}_{*}(x,\mu)&=v(x)(-1)^{k}S_{k}(D^{2}v(x))+\big(h_{0}(\mu-1)(-\ln v(x))^{-1}\\
&+1-\mu(\ln
v(x))^{-1}\big)\sum_{i=1}^{C_{N}^{k}}(-1)^{k}\mbox{det}(v_{x_{is}x_{ij}}(x))\\
&\times(-\nabla v_{i}(x))^{T}B(v_{i}(x))\nabla v_{i}(x),\,(x,
\mu)\in \bar{\Omega}\times [1, \mu_{*}],
\end{split}
\]
where $H_{0}$ and $h_{0}$ are given by \eqref{3f1} and \eqref{3f2}.
So,  we have
\[
\bigg(\frac{b_{1}}{\max\limits_{(x, \mu)\in\bar{\Omega}\times[0,
\mu_{*}]}\overline{\omega}_{*}(x,
\mu)}\bigg)^{1/k}\leq\tau_{3}\leq\tau_{4}\leq\bigg(\frac{b_{2}}{\min\limits_{(x,
\mu)\in\bar{\Omega}\times[0, \mu_{*}]}\underline{\omega}_{*}(x,
\mu)}\bigg)^{1/k}.
\]
The rest of the proof is similar to that of Theorem \ref{thm2.3} and
the proof is omitted here.
\end{proof}
\subsection{Proof of Theorem \ref{thm2.5}}
\begin{proof}
By Proposition \ref{HUR} $\mathbf{(iv)}$, we have
\begin{equation}\label{3f11}
\lim_{d(x)\rightarrow0}\frac{(v(x))^{\frac{k+1+\lambda}{k}}\tilde{L}(v(x))}{\int_{0}^{v(x)}s^{\frac{1+\lambda}{k}}\tilde{L}(s)ds}=\frac{k+1+\lambda}{k}.
\end{equation}
This implies that
\[
\lim_{d(x)\rightarrow0}\bigg(\frac{(v(x))^{\frac{k+1+\lambda}{k}}\tilde{L}(v(x))}{\int_{0}^{v(x)}s^{\frac{1+\lambda}{k}}\tilde{L}(s)ds}h_{0}-\frac{\lambda+1}{k}-\frac{v(x)\tilde{L}'(v(x))}{\tilde{L}(v(x))}\bigg)=\frac{kh_{0}+(1+\lambda)(h_{0}-1)}{k}.
\]
Since \eqref{f17} holds, without loss of generality, we always
assume that
\[
\min_{x\in\bar{\Omega}}\bigg(\frac{(v(x))^{\frac{k+1+\lambda}{k}}\tilde{L}(v(x))}{\int_{0}^{v(x)}s^{\frac{1+\lambda}{k}}\tilde{L}(s)ds}h_{0}-\frac{\lambda+1}{k}-\frac{v(x)\tilde{L}'(v(x))}{\tilde{L}(v(x))}\bigg)>0.
\]
By using this, combined with \eqref{3f1}-\eqref{3f3} and a similar
argument as in the proof of Theorem \ref{thm2.1}, we obtain that
there exist positive constants $\beta_{5}<\beta_{6}$ such that for
any $\tau>0$, it holds
\[
\beta_{5}<\omega_{4}(\tau, x)<\beta_{6},\,x\in\Omega,
\]
where $\omega_{4}$ is given by \eqref{2J4}.

By the similar argument as in the proof of Theorem \ref{thm2.3}, we
see that there exist positive constants $\tau_{5}$ and $\tau_{6}$
$(\tau_{5}\leq \tau_{6})$ such that \eqref{f18} holds.

Let
$\underline{u}(x)=\psi\big(\tau_{6}\int_{0}^{v(x)}s^{\frac{1+\lambda}{k}}\tilde{L}(s)ds\big),\,x\in\Omega$.
By $\mathbf{(b_{1})}$, $\mathbf{(b_{4})}$ and Lemma \ref{lemma2.3},
we obtain
\[
\begin{split}
S_{k}(D^{2}\underline{u}(x))&=\tau_{6}^{k}\bigg(-\psi'\bigg(\tau_{6}\int_{0}^{v(x)}s^{\frac{1+\lambda}{k}}\tilde{L}(s)ds\bigg)\bigg)^{k}(v(x))^{\lambda}\tilde{L}^{k}(v(x))\omega_{4}(\tau_{6},
x)\\
&\geq\tau_{6}^{k}\min_{x\in\bar{\Omega}}\omega_{4}(\tau_{6},
x)(v(x))^{\lambda}\tilde{L}^{k}(v(x))f(\underline{u}(x))\\
&=b_{2}(v(x))^{\lambda}\tilde{L}^{k}(v(x))f(\underline{u}(x))\geq
b(x)f(\underline{u}(x)),\,x\in\Omega,
\end{split}
\]
i.e., $\underline{u}$ is a subsolution to problem \eqref{M} in
$\Omega$. Moreover, by a straightforward calculation, we have
\[
\begin{split}
S_{l}(D^{2}\underline{u}(x))&=\tau_{6}^{l}\bigg(-\psi'\bigg(\tau_{6}\int_{0}^{v(x)}s^{\frac{1+\lambda}{k}}\tilde{L}(s)ds\bigg)\bigg)^{l}(v(x))^{\lambda}\tilde{L}^{l}(v(x))\\
&\times(v(x))^{\frac{(l-k)(1+\lambda)}{k}}\bigg[v(x)\cdot(-1)^{l}S_{l}(D^{2}v(x))+\bigg(\Psi\bigg(\tau_{6}\int_{0}^{v(x)}s^{\frac{1+\lambda}{k}}\tilde{L}(s)ds\bigg)\\
&\times\frac{(v(x))^{\frac{k+1+\lambda}{k}}\tilde{L}(v(x))}{\int_{0}^{v(x)}s^{\frac{1+\lambda}{k}}\tilde{L}(s)ds}-\frac{\lambda+1}{k}-\frac{v(x)\tilde{L}'(v(x))}{\tilde{L}(v(x))}\bigg)\\
&\times\sum_{i=1}^{C_{N}^{l}}(-1)^{l}det(v_{x_{is}{x_{ij}}}(x))(-\nabla
v_{i}(x))^{T}B(v_{i}(x))\nabla
v_{i}(x)\bigg]>0,\,x\in\Omega,\,l=1,\cdot\cdot\cdot,N.
\end{split}
\]
This implies that $\underline{u}$ is strictly convex in $\Omega$.

In a similar way, we can show that
$\overline{u}=\psi\big(\tau_{5}\int_{0}^{v}s^{\frac{1+\lambda}{k}}\tilde{L}(s)ds\big)$
is a strictly convex supersolution in $\Omega$. Obviously,
$\underline{u}\leq \overline{u}$ in $\Omega$. By Lemma
\ref{Lemma4.1}, we see that problem \eqref{M} has a classical
$k$-convex solution $u\in C^{\infty}(\Omega)$ satisfying
\begin{equation}\label{3f17}
\underline{u}(x)\leq u(x)\leq \overline{u}(x),\,x\in\Omega.
\end{equation}
\end{proof}
\subsection{Proof of Theorem \ref{thm2.6}}
\begin{proof}
By \eqref{3f11} and $\lambda<0$,  without loss of generality, we
assume that $v(x)$ satisfies
\begin{equation}\label{3f12}
\frac{(v(x))^{\frac{k+1+\lambda}{k}}\tilde{L}(v(x))}{(k+1)\int_{0}^{v(x)}s^{\frac{1+\lambda}{k}}\tilde{L}(s)ds}-\frac{\lambda+1}{k}-\frac{v(x)\tilde{L}(v(x))}{\tilde{L}(v(x))}>0,\,\forall
\,x\in\bar{\Omega}.
\end{equation}
On the other hand, by Lemma \ref{lemma29}
$\mathbf{(ii)}$-$\mathbf{(iv)}$, we see that
\begin{equation}\label{3f13}
0=\inf_{t>0}\Phi(t)\leq \sup_{t>0}\Phi(t)<+\infty.
\end{equation}
It follows from \eqref{3f12}-\eqref{3f13} that
\begin{equation}\label{3f15}
\sup_{(\tau, x)\in(0, 1]\times\Omega}\omega_{5}(\tau, x)<+\infty,
\end{equation}
where $\omega_{5}$ is given by \eqref{2J5}. Fix $\tau\in(0, 1]$. By
the similar argument as in the proof of Theorem \ref{thm2.1},  we
have
\[
\min_{x\in\bar{\Omega}}\omega_{5}(\tau, x)>0.
\]
This together with \eqref{3f15} implies that there exists
$\underline{\tau}\in(0, 1)$ such that
\begin{equation}\label{bpp}
\underline{\tau}^{k+1}\min_{x\in\bar{\Omega}}\omega_{5}(\underline{\tau},
x)<b_{2}.
\end{equation}

Take $\delta_{1}>0$ such that \eqref{3f14} is
positive in $\Omega_{\delta_{1}}$. We see from 
Lemma \ref{lemma29} $\mathbf{(ii)}$-$\mathbf{(iv)}$ that
\[
\begin{split}
0&<\inf_{\substack{(\tau, x)\in[1,
+\infty)\times\Omega\setminus\Omega_{\delta_{1}}}}\Phi\bigg(\tau\bigg(\int_{0}^{v(x)}s^{\frac{\lambda+1}{k}}\tilde{L}(s)ds\bigg)^{\frac{k}{k+1}}\bigg)\\
&\leq\sup_{\substack{(\tau, x)\in[1,
+\infty)\times\Omega\setminus\Omega_{\delta_{1}}}}\Phi\bigg(\tau\bigg(\int_{0}^{v(x)}s^{\frac{\lambda+1}{k}}\tilde{L}(s)ds\bigg)^{\frac{k}{k+1}}\bigg)<+\infty.
\end{split}
\]
This, combined with \eqref{3f12}, shows that
\begin{equation}\label{3f16}
0<\inf_{(\tau, x)\in[1, +\infty)\times\bar{\Omega}}\omega_{5}(\tau,
x)<\sup_{(\tau, x)\in[1, +\infty)\times\bar{\Omega}}\omega_{5}(\tau,
x)<+\infty.
\end{equation}
Fix $\tau\in[1, +\infty)$. By the similar argument as in the proof
of Theorem \ref{thm2.1},  we have
\[
\min_{x\in\bar{\Omega}}\omega_{5}(\tau, x)>0.
\]
This together with \eqref{3f16} implies that there exists
$\overline{\tau}\in [1, +\infty)$ such that
\begin{equation}\label{pbb}
\overline{\tau}^{k+1}\min_{x\in\bar{\Omega}}\omega_{5}(\overline{\tau},
x)>b_{2}.
\end{equation}
We conclude from \eqref{bpp}, \eqref{pbb} and Lemma \ref{Lemma4.2}
that there exists $\tau_{8}\in(\underline{\tau}, \overline{\tau})$
such that
\[
\tau_{8}^{k+1}\min_{x\in\bar{\Omega}}\omega_{5}(\tau_{8}, x)=b_{2}.
\]

By the similar argument as the above, we can show that there exists
positive constant $\tau_{7}$ with $\tau_{7}\leq \tau_{8}$ such that
\[
\tau_{7}^{k+1}\max_{x\in\bar{\Omega}}\omega_{5}(\tau_{7}, x)=b_{1}.
\]


Let
$\underline{u}(x)=\varphi\big(\tau_{8}(\int_{0}^{v(x)}s^{\frac{1+\lambda}{k}}\tilde{L}(s)ds)^{\frac{k}{k+1}}\big),\,x\in\Omega$.
By $\mathbf{(b_{1})}$, $\mathbf{(b_{4})}$ and Lemma \ref{lemma2.3},
we obtain
\[
\begin{split}
S_{k}(D^{2}\underline{u}(x))
&=\tau_{8}^{k+1}\bigg(-\varphi'\bigg(\tau\bigg(\int_{0}^{v(x)}s^{\frac{1+\lambda}{k}}\tilde{L}(s)ds\bigg)^{\frac{k}{k+1}}\bigg)\bigg)^{k-1}\\
&\times\varphi''\bigg(\tau\bigg(\int_{0}^{v(x)}s^{\frac{1+\lambda}{k}}\tilde{L}(s)ds\bigg)^{\frac{k}{k+1}}\bigg)\omega_{5}(\tau_{8},
x)\\
&\geq\tau_{8}^{k+1}\min_{x\in\bar{\Omega}}\omega_{5}(\tau_{8},
x)(v(x))^{\lambda}\tilde{L}^{k}(v(x))f(\underline{u}(x))\\
&=b_{2}(v(x))^{\lambda}\tilde{L}^{k}(v(x))f(\underline{u}(x))\geq
b(x)f(\underline{u}(x)),\,x\in\Omega,
\end{split}
\]
i.e., $\underline{u}$ is a subsolution to problem \eqref{M} in
$\Omega$. Moreover, by a straightforward calculation, we have
\[
\begin{split}
S_{l}(D^{2}\underline{u}(x))
&=\tau_{8}^{l+1}\bigg(\frac{k}{k+1}\bigg)^{k}\bigg(-\varphi'\bigg(\tau\bigg(\int_{0}^{v(x)}s^{\frac{1+\lambda}{k}}\tilde{L}(s)ds\bigg)^{\frac{k}{k+1}}\bigg)\bigg)^{l-1}\\
&\times\varphi''\bigg(\tau\bigg(\int_{0}^{v(x)}s^{\frac{1+\lambda}{k}}\tilde{L}(s)ds\bigg)^{\frac{k}{k+1}}\bigg)(v(x))^{\lambda}L^{l}(v(x))\\
&\times\bigg(\int_{0}^{v(x)}s^{\frac{1+\lambda}{k}}\tilde{L}(s)\bigg)^{\frac{k-l}{k+1}}(v(x))^{\frac{(l-k)(1+\lambda)}{k}}\bigg\{\Phi\bigg(\tau_{8}\bigg(\int_{0}^{v(x)}s^{\frac{1+\lambda}{k}}\tilde{L}(s)ds\bigg)^{\frac{k}{k+1}}\bigg)\\
&\times v(x)(-1)^{l}S_{l}(D^{2}v(x))+\bigg[\frac{(v(x))^{\frac{k+1+\lambda}{k}}\tilde{L}(v(x))}{\int_{0}^{v(x)}s^{\frac{1+\lambda}{k}}\tilde{L}(s)ds}+\Phi\bigg(\tau_{8}\bigg(\int_{0}^{v(x)}s^{\frac{1+\lambda}{k}}\tilde{L}(s)ds\bigg)^{\frac{k}{k+1}}\bigg)\\
&\times\bigg(\frac{(v(x))^{\frac{k+1+\lambda}{k}}\tilde{L}(v(x))}{(k+1)\int_{0}^{v(x)}s^{\frac{1+\lambda}{k}}\tilde{L}(s)ds}-\frac{\lambda+1}{k}-\frac{v(x)\tilde{L}'(v(x))}{\tilde{L}(v(x))}\bigg)\bigg]\\
&\times\sum_{i=1}^{C_{N}^{l}}(-1)^{l}det(
v_{x_{is}x_{ij}}(x))(-\nabla v_{i}(x))^{T}B(v_{i}(x))\nabla
v_{i}(x)\bigg\}>0,\,x\in\Omega,\,l=1,\cdot\cdot\cdot,N.
\end{split}
\]
This implies that $\underline{u}$ is strictly convex in $\Omega$.

In a similar way, we can show that
$\overline{u}=\varphi\big(\tau_{7}(\int_{0}^{v}s^{\frac{1+\lambda}{k}}\tilde{L}(s)ds)^{\frac{k}{k+1}}\big)$
is a strictly convex supersolution in $\Omega$. Obviously,
$\underline{u}\leq \overline{u}$ in $\Omega$. By Lemma
\ref{Lemma4.1}, we see that problem \eqref{M} has a classical
$k$-convex solution $u\in C^{\infty}(\Omega)$ satisfying
\eqref{3f17}.
\end{proof}
\section{The exact boundary behavior of large solutions}
In this section, we prove  Theorem \ref{thm2.8}.
\begin{proof}
Let $\varepsilon\in(0, b_{1}/2(1+C_{0}))$,
$C_{0}>\frac{b_{1}+b_{2}}{2}$ and
\[
\xi_{-\varepsilon}=\tau_{9}(1-(1+C_{0})\varepsilon/b_{1})^{1/k},\,\,\,\xi_{+\varepsilon}=\tau_{10}(1+(1+C_{0})\varepsilon/b_{2})^{1/k}.
\]
It follows that
\[
\tau_{9}\bigg(\frac{1}{2}\bigg)^{1/k}<\xi_{-\varepsilon}<\xi_{+\varepsilon}<\tau_{10}\bigg(\frac{3}{2}\bigg)^{1/k}.
\]
As \eqref{cj}, we define
\[
\Omega_{\delta_{*}}=\{x\in\Omega:0<d(x)<\delta_{*}\},
\]
where $\delta_{*}\in (0, \min\{\delta_{0},\,\tilde{\delta}\})$ and
$\tilde{\delta}$ is given in \eqref{gtg}.

Next, we consider the following two cases.

\noindent \textbf{Case 1.} $\theta$ is non-increasing on $(0,
\delta_{0})$. From Lemma \ref{lemma24} $\mathbf{(ii)}$ and Lemma
\ref{lemma28} $\mathbf{(i)}$, we see that corresponding to
$\varepsilon$, there exists sufficiently small constant
$\delta_{\varepsilon}\in(0, \delta_{*}/2)$ such that
$x\in\Omega_{2\delta_{\varepsilon}}$ and $r\in(0,
\delta_{\varepsilon})$, the following hold
\begin{equation}\label{3f18}
\frac{M_{k-1}^{-}}{1+\varepsilon}<S_{k-1}(\epsilon_{1},\cdot\cdot\cdot,\epsilon_{N-1})<\frac{M_{k-1}^{+}}{1-\varepsilon}\mbox{
and
}\frac{b_{1}-C_{0}\varepsilon}{1-\varepsilon}<\frac{b(x)}{\theta^{k+1}(d(x))}<\frac{b_{2}+C_{0}\varepsilon}{1+\varepsilon};
\end{equation}
\begin{equation}\label{3f19}
\mathfrak{X}_{1}(x,
\Theta_{r}^{\mp}(d(x))):=\frac{k+1}{k}\Psi\big(\xi_{\mp}(\Theta_{r}^{\mp}(d(x)))^{\frac{k+1}{k}}\big)-\frac{1}{k}
-\frac{\theta'(d(x))\Theta_{r}^{\mp}(d(x))}{\theta^{2}(d(x))}>0;
\end{equation}
\begin{equation}\label{3f21}
\begin{split}
\mathfrak{X}_{2}(x,
\Theta_{r}^{\mp}(d(x))):=&\bigg|\bigg(\frac{(k+1)\xi_{\mp}}{k}\bigg)^{k}\frac{\Theta(d(x))}{\theta(d(x))}
S_{k}(\epsilon_{1},\cdot\cdot\cdot,\epsilon_{N-1})(1\mp\varepsilon)\\
&+\bigg(\Psi\big(\xi^{\mp}(\Theta_{r}^{\mp}(d(x)))^{\frac{k+1}{k}}\big)-C_{f}^{+\infty}\bigg)M_{k-1}^{\pm}\\
&-\bigg(\frac{(k+1)\xi_{\mp}}{k}\bigg)^{k}\bigg(\frac{\theta'(d(x))\Theta(d(x))}{\theta^{2}(d(x))}-(1-D_{\theta})\bigg)M_{k-1}^{\pm}\bigg|<\varepsilon,
\end{split}
\end{equation}
where
\[
\Psi\big(\xi_{\mp}(\Theta_{r}^{\mp}(d(x)))^{\frac{k+1}{k}}\big)=-\frac{\psi''\big(\xi_{\mp}(\Theta_{r}^{\mp}(d(x)))^{\frac{k+1}{k}}\big)\xi_{\mp}(\Theta_{r}^{\mp}(d(x)))^{\frac{k+1}{k}}}{\psi'\big(\xi_{\mp}(\Theta_{r}^{\mp}(d(x)))^{\frac{k+1}{k}}\big)}
\]
and
\begin{equation}\label{3f22}
\Theta_{r}^{\mp}(d(x)):=\Theta(d(x))\mp\Theta(r)>0.
\end{equation}

Take $\sigma\in (0, \delta_{\varepsilon})$ and define
\begin{equation}\label{3f23}
D_{-}^{\sigma}:=\Omega_{2\delta_{\varepsilon}}\setminus\bar{\Omega}_{\sigma},\,\,\,D_{+}^{\sigma}:=\Omega_{2\delta_{\varepsilon}-\sigma}.
\end{equation}
Let
\[
\overline{u}_{\varepsilon}(x)=\psi\big(\xi_{-}(\Theta_{\sigma}^{-}(d(x)))^{\frac{k+1}{k}}\big),\,x\in
D_{-}^{\sigma} \mbox{ and
}\underline{u}_{\varepsilon}(x)=\psi\big(\xi_{+}(\Theta_{\sigma}^{+}(d(x)))^{\frac{k+1}{k}}\big),\,x\in
D_{+}^{\sigma}.
\]
By \eqref{3f18}-\eqref{3f21} and a straightforward calculation, we
have
\[
\begin{split}
&S_{k}(D^{2}\overline{u}_{\varepsilon}(x))-b(x)f(\overline{u}_{\varepsilon}(x))\\
&=\big(-\psi'\big(\xi_{-}(\Theta_{\sigma}^{-}(d(x)))^{\frac{k+1}{k}}\big)\big)^{k}\theta^{k+1}(d(x))\bigg\{\bigg(\frac{(k+1)\xi_{-}}{k}\bigg)^{k}\frac{\Theta_{\sigma}^{-}(d(x))}{\theta(d(x))}\\
&\times
S_{k}(\epsilon_{1},\cdot\cdot\cdot,\epsilon_{N-1})+\bigg(\frac{(k+1)\xi_{-}}{k}\bigg)^{k}\bigg[\frac{k+1}{k}\Psi\big(\xi_{-}(\Theta_{\sigma}^{-}(d(x)))^{\frac{k+1}{k}}\big)
-\frac{1}{k}-\frac{\theta'(d(x))\Theta_{\sigma}^{-}(d(x))}{\theta^{2}(d(x))}\bigg]\\
&\times
S_{k-1}(\epsilon_{1},\cdot\cdot\cdot,\epsilon_{N-1})-\frac{b(x)}{\theta^{k+1}(d(x))}\bigg\}\\
\end{split}
\]
\[
\begin{split}
&\leq
f(\overline{u}_{\varepsilon}(x))\theta^{k+1}(d(x))(1-\varepsilon)^{-1}\bigg[\bigg(\frac{(k+1)\xi_{-}}{k}\bigg)^{k}\frac{\Theta(d(x))}{\theta(d(x))}
S_{k}(\epsilon_{1},\cdot\cdot\cdot,\epsilon_{N-1})(1-\varepsilon)\\
&+\bigg(\frac{(k+1)\xi_{-}}{k}\bigg)^{k}\frac{k+1}{k}\big(\Psi\big(\xi_{-}(\Theta_{\sigma}^{-}(d(x)))^{\frac{k+1}{k}}\big)-C_{f}^{+\infty}\big)M_{k-1}^{+}-\bigg(\frac{(k+1)\xi_{-}}{k}\bigg)^{k}\\
&\times\bigg(\frac{\theta'(d(x))\Theta(d(x))}{\theta^{2}(d(x))}-(1-D_{\theta})\bigg)M_{k-1}^{+}+\bigg(\frac{(k+1)}{k}\bigg)^{k}\xi_{-}^{k}\frac{(k+1)C_{f}^{+\infty}+kD_{\theta}-(k+1)}{k}\\
&-(b_{1}-C_{0}\varepsilon)\bigg]\leq
f(u_{\varepsilon}(x))\theta^{k+1}(d(x))(1-\varepsilon)^{-1}\bigg[\bigg(\frac{(k+1)}{k}\bigg)^{k}\xi_{-}^{k}\frac{(k+1)C_{f}^{+\infty}+kD_{\theta}-(k+1)}{k}\\
&-b_{1}+(1+C_{0})\varepsilon\bigg]\leq0,
\end{split}
\]
i.e., $\overline{u}_{\varepsilon}$ is a supersolution to Eq.
\eqref{M} in $D_{-}^{\sigma}$. Moreover, by \eqref{3f19}, we see
that
\[
\begin{split}
&S_{i}(D^{2}\overline{u}_{\varepsilon}(x))\\
&=\big(-\psi'\big(\xi_{-}(\Theta_{\sigma}^{-}(d(x)))^{\frac{k+1}{k}}\big)\big)^{i}(\Theta_{\sigma}^{-}(d(x)))^{\frac{i-k}{k}}\theta^{i+1}(d(x))\bigg[\bigg(\frac{(k+1)\xi_{-}}{k}\bigg)^{i}\frac{\Theta_{\sigma}^{-}(d(x))}{\theta(d(x))}\\
&\times
S_{i}(\epsilon_{1},\cdot\cdot\cdot,\epsilon_{N-1})\bigg]+\bigg(\frac{(k+1)\xi_{-}}{k}\bigg)^{i}\bigg[\frac{k+1}{k}\Psi\big(\xi_{-}(\Theta_{\sigma}^{-}(d(x)))^{\frac{k+1}{k}}\big)
-\frac{1}{k}-\frac{\theta'(d(x))\Theta_{\sigma}^{-}(d(x))}{\theta^{2}(d(x))}\bigg]\\
&\times
S_{i-1}(\epsilon_{1},\cdot\cdot\cdot,\epsilon_{N-1})>0,\,x\in
D_{-}^{\sigma},\,i=1,\cdot\cdot\cdot, N.
\end{split}
\]
This implies that $\overline{u}_{\varepsilon}$ is strictly convex in
$D_{-}^{\sigma}$.

In a similar way, we can show that $\underline{u}_{\varepsilon}$ is
a strictly convex subsolution to Eq. \eqref{M} in $D_{+}^{\sigma}$.

\noindent \textbf{Case 2.} $\theta$ is non-decreasing on $(0,
\delta_{0})$. Since $\lim_{t\rightarrow0^{+}}\Theta(t)=0$, for
convenience, we define
\[
\Theta_{0}^{\mp}(d(x)):=\Theta(d(x))
\]
in \eqref{3f22}. From Lemma \ref{lemma24} $\mathbf{(ii)}$ and Lemma
\ref{lemma28} $\mathbf{(i)}$, we see that corresponding to
$\varepsilon$, there exists sufficiently small constant
$\delta_{\varepsilon}\in(0, \delta_{*}/2)$ such that for
$x\in\Omega_{2\delta_{\varepsilon}}$, \eqref{3f18}-\eqref{3f21} with
$r=0$ hold here.

Take $\sigma\in (0, \delta_{\varepsilon})$ and define
\[
\overline{u}_{\varepsilon}(x)=\psi\big(\xi_{-}(\Theta(d_{-}(x)))^{\frac{k+1}{k}}\big),\,x\in
D_{-}^{\sigma} \mbox{ and
}\underline{u}_{\varepsilon}(x)=\psi\big(\xi_{+}(\Theta(d_{+}(x)))^{\frac{k+1}{k}}\big),\,x\in
D_{+}^{\sigma},
\]
where $D_{\mp}^{\sigma}$ are defined as shown in \eqref{3f23} and
$d_{-}(x):=d(x)-\sigma,\,d_{+}(x):=d(x)+\sigma$.  A straightforward
calculation shows that
\[
\begin{split}
&S_{k}(D^{2}\overline{u}_{\varepsilon}(x))-b(x)f(\overline{u}_{\varepsilon}(x))\\
&=\big(-\psi'\big(\xi_{-}(\Theta(d_{-}(x)))^{\frac{k+1}{k}}\big)\big)^{k}\theta^{k+1}(d_{-}(x))\bigg\{\bigg(\frac{(k+1)\xi_{-}}{k}\bigg)^{k}\frac{\Theta(d_{-}(x))}{\theta(d_{-}(x))}\\
&\times
S_{k}(\epsilon_{1},\cdot\cdot\cdot,\epsilon_{N-1})+\bigg(\frac{(k+1)\xi_{-}}{k}\bigg)^{k}\bigg[\frac{k+1}{k}\Psi\big(\xi_{-}(\Theta(d_{-}(x)))^{\frac{k+1}{k}}\big)
-\frac{1}{k}-\frac{\theta'(d_{-}(x))\Theta(d_{-}(x))}{\theta^{2}(d_{-}(x))}\bigg]\\
&\times
S_{k-1}(\epsilon_{1},\cdot\cdot\cdot,\epsilon_{N-1})-\frac{b(x)}{\theta^{k+1}(d_{-}(x))}\bigg\}\\
&\leq
f(\overline{u}_{\varepsilon}(x))\theta^{k+1}(d_{-}(x))(1-\varepsilon)^{-1}\bigg[\bigg(\frac{(k+1)\xi_{-}}{k}\bigg)^{k}\frac{\Theta(d_{-}(x))}{\theta(d_{-}(x))}
S_{k}(\epsilon_{1},\cdot\cdot\cdot,\epsilon_{N-1})(1-\varepsilon)\\
\end{split}
\]
\[
\begin{split}
&+\bigg(\frac{(k+1)\xi_{-}}{k}\bigg)^{k}\frac{k+1}{k}\big(\Psi\big(\xi_{-}(\Theta(d_{-}(x)))^{\frac{k+1}{k}}\big)-C_{f}^{+\infty}\big)M_{k-1}^{+}-\bigg(\frac{(k+1)\xi_{-}}{k}\bigg)^{k}\\
&\times\bigg(\frac{\theta'(d_{-}(x))\Theta(d_{-}(x))}{\theta^{2}(d_{-}(x))}-(1-D_{\theta})\bigg)M_{k-1}^{+}+\bigg(\frac{(k+1)}{k}\bigg)^{k}\xi_{-}^{k}\frac{(k+1)C_{f}^{+\infty}+kD_{\theta}-(k+1)}{k}\\
&-(b_{1}-C_{0}\varepsilon)\bigg]\leq
f(\overline{u}_{\varepsilon}(x))\theta^{k+1}(d_{-}(x))(1-\varepsilon)^{-1}\bigg[\bigg(\frac{(k+1)}{k}\bigg)^{k}\xi_{-}^{k}\frac{(k+1)C_{f}^{+\infty}+kD_{\theta}-(k+1)}{k}\\
&-b_{1}+(1+C_{0})\varepsilon\bigg]\leq0,
\end{split}
\]
i.e., $\overline{u}_{\varepsilon}$ is a supersolution to Eq.
\eqref{M} in $D_{-}^{\sigma}$.  By the similar argument as the
above, we can show $\overline{u}_{\varepsilon}$ is strictly convex
in $D_{-}^{\sigma}$.

In a similar way, we can show that $\underline{u}_{\varepsilon}$ is
a strictly convex subsolution to Eq. \eqref{M} in $D_{+}^{\sigma}$.

For Case $1$ and Case 2, let $u$ be an arbitrary $k$-convex solution
to problem \eqref{M}. We assert that there exists a large constant
$M>0$ such that
\begin{equation}\label{3f26}
u(x)\leq \overline{u}_{\varepsilon}(x)+M,\,x\in D_{-}^{\sigma}
\mbox{ and }\underline{u}(x)\leq u(x)+M,\,x\in D_{+}^{\sigma}.
\end{equation}
In fact, we can take a constant $M>0$ independent of $\sigma$ such
that
\begin{equation}\label{3f24}
\begin{split}
&u(x)\leq
\overline{u}_{\varepsilon}(x)+M,\,x\in\{x\in\Omega:d(x)=2\delta_{\varepsilon}\},\\
&\underline{u}_{\varepsilon}(x)\leq u(x)+M,\,x\in\{x\in\Omega:
d(x)=2\delta_{\varepsilon}-\sigma\}.
\end{split}
\end{equation}
Moreover, we also see that
\[
u(x)<\overline{u}_{\varepsilon}(x)=+\infty,\,x\in\{x\in\Omega:d(x)=\sigma\}\mbox{
and
}\underline{u}_{\varepsilon}(x)<u(x)=+\infty,\,x\in\partial\Omega.
\]
This implies that we can take a sufficiently small
$0<\rho<\delta_{\varepsilon}$ such that
\begin{equation}\label{3f25}
\sup_{x\in D_{-}^{\sigma}}u(x)\leq
\overline{u}_{\varepsilon}(x),\,x\in D_{-}^{\sigma}\setminus
\tilde{D}_{-}^{\sigma}\mbox{ and }\sup_{x\in
D^{\sigma}_{+}}\underline{u}_{\varepsilon}(x)\leq
u(x),\,D_{+}^{\sigma}\setminus \tilde{D}_{+}^{\sigma},
\end{equation}
where
\[
\tilde{D}_{-}^{\sigma}=\Omega_{2\delta_{\varepsilon}}\setminus\bar{\Omega}_{(1+\rho)\sigma}
\mbox{ and
}\tilde{D}_{+}^{\sigma}=\Omega_{2\delta_{\varepsilon}-\sigma}\setminus\bar{\Omega}_{\rho}.
\]
By $\mathbf{(S_{1})}$ (or $\mathbf{(S_{01})}$), we note that
$\overline{u}_{\varepsilon}+M$ and $u+M$ are both supersolution in
$\tilde{D}_{-}^{\sigma}$ and $\tilde{D}_{+}^{\sigma}$, respectively.
It follows from \eqref{3f24}-\eqref{3f25} and Lemma \ref{lemma2.1}
that
\[
u(x)\leq \overline{u}_{\varepsilon}(x)+M,\,x\in
\tilde{D}_{-}^{\sigma} \mbox{ and }
\underline{u}_{\varepsilon}(x)\leq u(x)+M,\,x\in
\tilde{D}_{+}^{\sigma}.
\]
This together with \eqref{3f25} implies that \eqref{3f26} holds.
Hence, letting $\sigma\rightarrow0$, we have for $x\in
\Omega_{2\delta_{\varepsilon}}$,
\[
\frac{u(x)}{\psi\big(\xi_{-}(\Theta(d(x)))^{\frac{k+1}{k}}\big)}\leq
1+\frac{M}{\psi\big(\xi_{-}(\Theta(d(x)))^{\frac{k+1}{k}}\big)}
\]
and
\[
\frac{u(x)}{\psi\big(\xi_{+}(\Theta(d(x)))^{\frac{k+1}{k}}\big)}\geq
1+\frac{M}{\psi\big(\xi_{+}(\Theta(d(x)))^{\frac{k+1}{k}}\big)}.
\]
Consequently,  we have
\[
\limsup_{d(x)\rightarrow0}\frac{u(x)}{\psi\big(\xi_{-}(\Theta(d(x)))^{\frac{k+1}{k}}\big)}\leq1
\mbox{ and
}\liminf_{d(x)\rightarrow0}\frac{u(x)}{\psi\big(\xi_{+}(\Theta(d(x)))^{\frac{k+1}{k}}\big)}\geq1.
\]
It follows from Lemma \ref{lemma24} $\mathbf{(ii)}$ that
\[
\limsup_{d(x)\rightarrow0}\frac{u(x)}{\psi\big((\Theta(d(x)))^{\frac{k+1}{k}}\big)}=\limsup_{d(x)\rightarrow0}\frac{u(x)}{\psi\big(\xi_{-}(\Theta(d(x)))^{\frac{k+1}{k}}\big)}\lim_{d(x)\rightarrow0}\frac{\psi\big(\xi_{-}(\Theta(d(x)))^{\frac{k+1}{k}}\big)}{\psi\big((\Theta(d(x)))^{\frac{k+1}{k}}\big)}\leq\xi_{-}^{1-C_{f}^{+\infty}}
\]
and
\[
\liminf_{d(x)\rightarrow0}\frac{u(x)}{\psi\big((\Theta(d(x)))^{\frac{k+1}{k}}\big)}=\liminf_{d(x)\rightarrow0}\frac{u(x)}{\psi\big(\xi_{+}(\Theta(d(x)))^{\frac{k+1}{k}}\big)}\lim_{d(x)\rightarrow0}\frac{\psi\big(\xi_{+}(\Theta(d(x)))^{\frac{k+1}{k}}\big)}{\psi\big((\Theta(d(x)))^{\frac{k+1}{k}}\big)}\geq\xi_{+}^{1-C_{f}^{+\infty}}.
\]
Passing to $\varepsilon\rightarrow0$, we obtain \eqref{3f27}
holds.
\end{proof}
\section*{Acknowledgments} We would like to thank Professor Zhijun Zhang for his encouragement and many suggestions in
this project.

\medskip

\medskip

\end{document}